\documentclass{amsart}

\usepackage{todonotes}
\usepackage{macros}
\usepackage{bm}
\usepackage{algorithm,algpseudocode}
\usepackage{caption,subcaption}
\usepackage{tikz,pgfplots}
\usepgfplotslibrary{groupplots}
\usetikzlibrary{matrix}
\usepackage{multirow}
\usepackage[margin=1.3in]{geometry}

\usepackage{txfonts}

\numberwithin{equation}{section}

\newtheorem{theorem}{Theorem}
\newtheorem{remark}{Remark}
\newtheorem{lemma}{Lemma}
\newtheorem{corollary}{Corollary}

\begin{document}

\title[Quadrature-Free Polytopic DGFEMs for Transport Problems]{Quadrature-Free Polytopic Discontinuous Galerkin Methods for Transport Problems}

\author[T. J. Radley]{Thomas J. Radley}
\address[T. J. Radley]{School of Mathematical Sciences, University of Nottingham,
University Park, Nottingham NG7 2RD, UK}
\email{Thomas.Radley1@nottingham.ac.uk}
\author[P. Houston]{Paul Houston}
\address[P. Houston]{School of Mathematical Sciences, University of Nottingham,
University Park, Nottingham NG7 2RD, UK}
\email[Corresponding author]{Paul.Houston@nottingham.ac.uk}
\author[M. E. Hubbard]{Matthew E. Hubbard}
\address[M. E. Hubbard]{School of Mathematical Sciences, University of Nottingham,
University Park, Nottingham NG7 2RD, UK}
\email{Matthew.Hubbard@nottingham.ac.uk}

\begin{abstract}
In this article we consider the application of Euler's homogeneous function theorem together with Stokes' theorem to exactly integrate families of polynomial spaces over general polygonal and polyhedral (polytopic) domains in two- and three-dimensions, respectively. This approach allows for the integrals to be evaluated based on only computing the values of the integrand and its derivatives at the vertices of the polytopic domain, without the need to construct a sub-tessellation of the underlying domain of interest. Here, we present a detailed analysis of the computational complexity of the proposed algorithm and show that this depends on three key factors: the ambient dimension of the underlying polytopic domain; the size of the requested polynomial space to be integrated; and the size of a directed graph related to the polytopic domain. This general approach is then employed to compute the volume integrals arising within the discontinuous Galerkin finite element approximation of the linear transport equation. Numerical experiments are presented which highlight the efficiency of the proposed algorithm when compared to standard quadrature approaches defined on a sub-tessellation of the polytopic elements.

\textbf{Keywords:} Polytopic elements; numerical integration; discontinuous Galerkin methods.

\textbf{Mathematics Subject Classification (2020):} 65D30, 65Y20.
\end{abstract}

\maketitle

\section{Introduction}

Over the past 10-15 years there has been increasing widespread interest in the development of numerical methods for the approximation of partial differential equations based on employing general element shapes such as polygons in 2D and polyhedra in 3D (which we collectively refer to as polytopic elements), cf. \cite{Antonietti2016,doi:10.1142/S0218202512500492,cangiani2014hp,CaDoGeHo2017,hho_book_2021}, for example, and the references cited therein. Exploiting such flexible computational meshes is highly attractive for a number of key reasons: complicated geometries may be meshed by employing relatively coarse meshes, without the need to simplify the boundary or other features within the given domain; moving meshes and/or overlapping meshes arising, for example, in applications such as fluid structure interaction can easily be accommodated; hanging nodes are treated in a very simple and natural manner; finally, multi-level solvers, such as domain decomposition methods and multigrid can easily be implemented employing embedded polytopic meshes. 

However, a major bottleneck in the exploitation of general polytopic elements within, for example, finite element methods is the construction of suitable element quadratures needed for the assembly of the underlying matrix system. With this mind several approaches have been proposed within the literature. The simplest approach is to construct a sub-tessellation of a given polytope into standard element shapes, for example, simplices and/or tensor-product elements (quadrilaterals in 2D and hexahedra in 3D) and employ known quadrature rules on these sub-elements. We remark that when the underlying polytopic mesh is constructed from the agglomeration of a given fine mesh consisting of standard element shapes, then the sub-tessellation need not be computed, as it will already be provided. The major problem with this approach is that the number of quadrature points can be extremely large, depending on the cardinality of the sub-tessellation and the required accuracy. Of course, quadrature is naturally highly parallisable, and hence such an implementation can be accelerated, cf. ~\cite{doi:10.1137/20M1350984} for example, who employed a GPU approach. Alternatively, one may for example, use this initial quadrature and attempt to optimise it by successively removing points until a minimal number is attained; here, we mention the moment fitting approaches, cf. ~\cite{LyMo77,mousavi2010generalized,MoSu11,SuWa13,XiGi10}, for example, and the references cited therein.

In this article we pursue an alternative approach based on the integration of homogeneous functions. This idea was first developed in the articles \cite{chin2015numerical,lasserre1999integration}, cf. also \cite{chin2020efficient}. Here, the essential idea is to employ Euler's homogeneous function theorem, together with Stokes' theorem, which allows for the integral of a homogeneous function over a given polytope to be written as a boundary integral. Recursively applying this approach allows for the integral to be exactly evaluated based on only computing values of the integrand and its derivatives at the vertices of the polytopic domain, and hence leads to an exact cubature rule whose quadrature points are the vertices of the polytope. In our recent paper \cite{antonietti2018fast} we considered the application of this algorithm within the discontinuous Galerkin finite element (DGFEM) approximation of the Poisson problem. In this article, we extend this work by first presenting a detailed complexity analysis of both a quadrature based algorithm and the proposed quadrature-free approach. Here, the primary focus is to consider the number of flops required to integrate an entire space of polynomials up to a given fixed degree, which is typically required within a finite element implementation. In particular, we show that the computational complexity of the proposed algorithm depends on three key factors: the ambient dimension of the underlying polytopic domain; the size of the requested polynomial space to be integrated; and the size of a directed graph related to the polytopic domain. By a careful selection of the so-called local origins needed on each of the polytope's lower dimensional facets, which is needed in the specification of the quadrature-free algorithm, we are able to considerably optimise the number of required flops within this approach by reducing the size of the associated directed graph. To demonstrate the application of this optimised quadrature-free algorithm, we utilise this within a DGFEM approximation of the linear transport equation. We point out that the specific application we have in mind is the numerical approximation of the linear Boltzmann transport equation. This is a high-dimensional integro-differential equation employed within applications including neutron and radiation transport. In our recent article \cite{houston2023efficient}, we illustrate that by employing a judicious choice of local basis functions and quadratures within the angular and energy domains, the DGFEM approximation of the underlying problem may be computed by simply solving a sequence of linear transport equations in space corresponding to each angular quadrature point (which defines a constant advection direction) and each energy quadrature point. In typical 3D applications utilising either source iteration, or a source iteration preconditioner within a Krylov space solver, means that an extremely large number of these transport solves must be computed. With that in mind, efficiency is of paramount concern, and hence the need to develop optimised quadrature-free implementations of the linear transport equation. Numerical experiments presented here highlight the gains in efficiency with our proposed optimised quadrature-free algorithm in comparison with standard sub-tessellation based quadrature. In particular, we study the dependence of both algorithms on the shape of the underlying polytope.

The outline of this article is as follows. In Section~\ref{section:monomial_integration_quad_free} we recall the quadrature-free integration technique introduced in ~\cite{chin2015numerical} and its application to the integration of families of monomial functions over general $d$-dimensional polytopes as in ~\cite{antonietti2018fast}. The time and space (memory) complexities of this algorithm in the case where the integration domain is polygonal or polyhedral are analysed. In Section~\ref{sec:dgfem} we introduce the DGFEM approximation of linear first-order transport equation on general polytopic methods. In Section~\ref{sec:matrix_assembly} we compare the time complexities associated with the assembly of the local DGFEM element matrices via quadrature based and quadrature-free based approaches, where we exploit known decomposition of the integrands into a linear combination of monomials. Several numerical experiments are performed in Section~\ref{sec:numerics} in order to demonstrate the accelerated assembly time of quadrature-free based approach. Finally, in Section~\ref{sec:conclusions} we summarise the work undertaken in this article and discuss future extensions.

\section{Quadrature-free integration of monomials} \label{section:monomial_integration_quad_free}

In this section we review the procedure for the numerical integration of homogeneous functions over a polytopic domain as introduced by Lasserre \cite{lasserre1998integration,lasserre1999integration} for convex polytopes and extended by Chin et al. in \cite{chin2015numerical} to non-convex ones.

\noindent We consider the problem of evaluating integrals of the form $\int_\polytope f(\x) \d\x$, where:
\begin{itemize}
    \item $f:\polytope\rightarrow\reals$ denotes a \emph{(positively) homogeneous function} of degree $q\in\reals$; that is, for all $\lambda>0$ we have that $f(\lambda\x)=\lambda^qf(\x)$ for all $\x\in\polytope$. The function $f$ satisfies Euler's homogeneous function theorem \cite{simon1994mathematics}: if $f$ is a continuously differentiable positively homogeneous function of degree $q$, then we have
    \begin{equation*}
        q f(\x) = \x\cdot\nabla f(\x)
    \end{equation*}

    \noindent for all $\x$ in the domain of definition of $f$. 
    
    \item $\polytope\subset\reals^d$, $d=2, 3$, denotes a closed polytope whose boundary $\partial\polytope$ is defined by $m(\polytope)$ polytopic facets $\{\facet_i\}_{i=1}^{m(\polytope)}$ of dimension $(d-1)$. To each facet $\facet_i$, we associate the hyperplane $\mathcal{H}_i$ containing $\facet_i$ defined for any $\x_0\in\reals^d$ by
    \begin{equation*}
        \mathcal{H}_i = \{ \x\in\reals^d : (\x-\x_0)\cdot\mathbf{n}_i = a_i \}
    \end{equation*}

    \noindent for some $a_i\in\reals$ and some vector $\mathbf{n}_i\in\reals^d$ of unit length. We note that $\mathbf{n}_i$ may be chosen to be the unit outward normal to $\polytope$ on $\facet_i$ and that $a_i$ denotes the (signed) distance between $\mathcal{H}_i$ and $\x_0$.
\end{itemize}

We also recall the generalised Stokes' theorem \cite{taylor1996partial}: for a continuously differentiable vector field $\mathbf{X}:\polytope\rightarrow\reals^d$ defined on a neighbourhood of $\polytope$, we have
\begin{equation*}
    \int_\polytope \nabla\cdot\mathbf{X} \d\x = \int_{\partial\polytope} \mathbf{X}\cdot\mathbf{n} \d\sigma,
\end{equation*}

\noindent where $\d\sigma$ denotes the Lebesgue measure on $\partial\polytope$ and $\mathbf{n}$ denotes the unit outward normal to $\polytope$ on $\partial\polytope$.

By setting $\mathbf{X}=(\x-\x_0)f(\x)$ and invoking Euler's homogeneous function theorem, it can be shown that
\begin{equation} \label{eqn:volume_face_relation_general}
    \int_\polytope f(\x) \d\x = \frac{1}{d+q} \left[ \sum_{i=1}^{m(\polytope)} a_i \int_{\facet_i} f(\x) \d\sigma + \int_\polytope \x_0\cdot\nabla f(\x) \ \d\x \right].
\end{equation}

\noindent Equation \reff{eqn:volume_face_relation_general} relates integration of $f$ over $\polytope$ in terms of integration of $f$ over the facets of $\polytope$, as well as the integration of $\nabla f$ over $\polytope$. By selecting $\x_0=\bm{0}$, \reff{eqn:volume_face_relation_general} reduces to the more common expression relating the integrals of $f$ over $\polytope$ and $\partial\polytope$ \cite{antonietti2018fast,chin2015numerical,chin2020efficient,lasserre1999integration}:
\begin{equation*}
    \int_\polytope f(\x) \d\x = \frac{1}{d+q} \sum_{i=1}^{m(\polytope)} a_i \int_{\facet_i} f(\x) \d\sigma.
\end{equation*}

As shown in \cite{antonietti2018fast,chin2015numerical,lasserre1998integration,lasserre1999integration}, for example, one may recursively apply the generalised Stokes' theorem to express the integral $\int_{\facet_i} f(\x) \d\sigma$ in terms of integrals over the $(d-2)$-dimensional boundary facets $\{\facet_{ij}\}_{j=1}^{m(\facet_i)}$ of $\facet_i$:
\begin{equation} \label{eqn:face_edge_relation_general}
    \int_{\facet_i} f(\x) \d\sigma = \frac{1}{d-1+q} \left[ \sum_{j=1}^{m(\facet_i)} a_{ij} \int_{\facet_{ij}} f(\x) \ \d\nu + \int_{\facet_i} \x_1\cdot\nabla f(\x) \ \d\sigma \right],
\end{equation}

\noindent where $\d\nu$ denotes the Lebesgue measure on $\partial\facet_i$, $\x_1\in\mathcal{H}_i$ is arbitrary and $a_{ij}$ denotes the Euclidean distance from $\x_1$ to $\facet_{ij}$.

Equations \reff{eqn:volume_face_relation_general} and \reff{eqn:face_edge_relation_general} can be generalised to give the integral of $f$ over any $k$-dimensional facet $\facet$, $0\le k\le d$, in terms of the integral of the same function over the boundary $\partial\facet=\{\partial\facet_i\}_{i=1}^{m(\facet)}$ and the integral of $\nabla f$ over $\facet$:
\begin{equation} \label{eqn:facet_to_subfacet_general}
    \int_\facet f(\x) \d s = \frac{1}{\dim\facet+q} \left[ \sum_{i=1}^{m(\facet)} \dist(\partial\facet_i,\x_\facet) \int_{\partial\facet_i} f(\x) \d\xi + \int_\facet \x_\facet\cdot\nabla f(\x) \d s \right],
\end{equation}

\noindent where $\d s$ (respectively $\d\xi$) denotes the $k$-dimensional (respectively $(k-1)$-dimensional) Lebesgue measure on $\facet$ (respectively $\partial\facet$), $\x_\facet$ is an arbitrary point contained in $\facet$ (or the $k$-dimensional hyperplane containing $\facet$), and $\dist(\partial\facet_i,\x_\facet)$ denotes the Euclidean distance from $\x_\facet$ to (the $k$-dimensional hyperplane containing) $\partial\facet_i$. Finally, in the case where $\facet=\x_\facet\in\reals^d$ (that is, $\dim\facet=0$), the right-hand-side of \reff{eqn:facet_to_subfacet_general} can be replaced with the point evaluation $f(\x_\facet)$.

In the case where $f(\x)=\x^{\bm{\alpha}}=\prod_{k=1}^d x_k^{\alpha_k}$ is a monomial function in $d$ variables, \reff{eqn:facet_to_subfacet_general} gives the following recursive formula for the integrals
\begin{equation} \label{eqn:facet_to_subfacet_monomial}
    \ifun(\facet,\bm{\alpha}) := \int_\facet \x^{\bm{\alpha}} \d s = \frac{1}{\dim\facet+|\bm{\alpha}|} \left[ \sum_{i=1}^{m(\facet)} \dist(\partial\facet_i,\x_\facet) \ifun(\partial\facet_i,\bm{\alpha}) + \sum_{j=1}^d \alpha_j (\x_\facet)_j \ifun(\facet,\bm{\alpha}-\mathbf{e}_j) \right],
\end{equation}
where $\mathbf{e}_i$, $i=1,2,\ldots,d$, denote the standard unit basis vectors in $\reals^d$.

Equation \reff{eqn:facet_to_subfacet_monomial} can be used to generate sets of integrals of monomial functions over $\polytope$. To this end, we consider the problem of evaluating the following set of integrals:
\begin{equation*}
    \ifun(\facet,\mathcal{J}) = \left\{ \ifun(\facet,\bm{\alpha}) = \int_\facet \x^{\bm{\alpha}} \d\x : \bm{\alpha}\in\mathcal{J} \right\}.
\end{equation*}

\noindent Here, $\mathcal{J}\subset\naturals^d$ denotes a set of multi-indices satisfying the following property: for each $\bm{\alpha}\in\mathcal{J}$ and $1\le i\le d$, we either have $\alpha_i=0$ or $\bm{\alpha}-\mathbf{e}_i\in\mathcal{J}$. This motivates the definition of Algorithm~\ref{alg:monomial_integration}, cf. \cite{antonietti2018fast}.
\begin{algorithm}[t!]
\caption{Evaluation of the set $\ifun(\facet,\mathcal{J})$ for a given $k$-dimensional facet $\facet$, $0\le k\le d$.}
\label{alg:monomial_integration}
\begin{algorithmic}[1]
    \Procedure{ComputeIntegrals}{$\facet$,$\mathcal{J}$}
        \State $\ifun(\facet,\bm{\alpha})\gets0$ for all $\bm{\alpha}\in\mathcal{J}$
        \State Select $\x_\facet$ as any point in (the $(\dim\facet)$-dimensional hyperplane containing) $\facet$
        \State Get boundary facets $\{\partial\facet_i\}_{i=1}^{m(\facet)}$
        \For{$i=1,\dots,m(\facet)$}
            \If{$\dist(\partial\facet_i,\x_\facet)\ne0$ and $\ifun(\partial\facet_i,\mathcal{J})$ not already computed}
                \State $\ifun(\partial\facet_i,\mathcal{J})\gets\textsc{ComputeIntegrals}(\partial\facet_i,\mathcal{J})$
            \EndIf
        \EndFor
        \For{$\bm{\alpha}\in\mathcal{J}$}
            \State \begin{equation*}
                    \ifun(\facet,\bm{\alpha}) \gets \frac{1}{\dim\facet+|\bm{\alpha}|} \Bigg[ \sum_{i=1}^{m(\facet)} \dist(\partial\facet_i,\x_\facet) \ifun(\partial\facet_i,\bm{\alpha}) + \sum_{j=1}^d \alpha_j (\x_\facet)_j \ifun(\facet,\bm{\alpha}-\mathbf{e}_j) \Bigg].
            \end{equation*}
        \EndFor
        \State \Return $\ifun(\facet,\mathcal{J})$
    \EndProcedure
\end{algorithmic}
\end{algorithm}

\begin{remark}[Termination of Algorithm \ref{alg:monomial_integration}]\label{remark:monomial_integration_termination}
    The recursion in Algorithm \ref{alg:monomial_integration} terminates when \linebreak$\textsc{ComputeIntegrals}(\facet,\mathcal{J})$ is called with $\dim\facet=0$; that is, $\facet=\x_\facet$ is a single point. Since $\partial\facet=\x_\facet$, the recursive function call in line 7 will not be executed.
\end{remark}

\subsection{Analysis of quadrature-free monomial integration algorithm} \label{section:monomial_integration_quad_free_analysis_subsec}

The computational complexity of Algorithm \ref{alg:monomial_integration} can be understood in terms of the size of the requested monomial set $\mathcal{J}$, as well as the complexity of the domain of integration $\polytope$.

\begin{theorem}[Time and space complexity of Algorithm \ref{alg:monomial_integration}]\label{thm:monomial_integral_alg_complexity}
    Assuming that the set
    \begin{equation*}
        \left\{ \dist(\partial\facet_i,\x_\facet) : 1\le i\le m(\facet), 0\le\dim\facet\le d \right\}
    \end{equation*}
    
    \noindent is pre-computed, the time complexity of Algorithm \ref{alg:monomial_integration}, measured as the total number of floating-point operations required to assemble $\ifun(\polytope,\mathcal{J})$, is $\bigO(\chi_1(\polytope)|\mathcal{J}|)$, where $|\mathcal{J}|$ denotes the number of requested monomial integrals and
    \begin{equation*}
        \chi_1(\polytope) = \sum_{k=0}^{d} \sum_{\substack{\facet\subseteq\polytope \\ \dim\facet=k}} (m(\facet)+d).
    \end{equation*}

    \noindent The space complexity of Algorithm 1, measured as the total number of floating-point numbers required to store $\ifun(\polytope,\mathcal{J})$, is $\bigO(\chi_2(\polytope)|\mathcal{J}|)$, where
    \begin{equation*}
        \chi_2(\polytope) = \sum_{k=0}^{d} \sum_{\substack{\facet\subseteq\polytope \\ \dim\facet=k}} 1.
    \end{equation*}
\end{theorem}

\begin{proof}
    We first analyse the number of floating-point operations required to compute the right-hand side of \reff{eqn:facet_to_subfacet_monomial} for a single facet $\facet$:
    \begin{itemize}
        \item the sum $S_1=\sum_{i=1}^{m(\facet)} \dist(\partial\facet_i,\x_\facet) \ifun(\partial\facet_i,\bm{\alpha})$ requires $m(\facet)$ products and $m(\facet)-1$ additions;
        \item the sum $S_2 = \sum_{j=1}^d \alpha_j(\x_\facet)_j\ifun(\facet,\bm{\alpha}-\mathbf{e}_j)$ requires $2d$ products and $d-1$ additions;
        \item the sum $S_3=\dim\facet+|\bm{\alpha}|=\dim\facet+\sum_{j=1}^d \alpha_j$ requires $d$ additions;
        \item the final result $\ifun(\facet,\bm{\alpha})=\frac{S_1+S_2}{S_3}$ requires one addition and one division.
    \end{itemize}

    \noindent We deduce that the right-hand-side of \reff{eqn:facet_to_subfacet_monomial} may be computed in $2m(\facet)+4d$ floating-point operations. Since this operation is called for each $\bm{\alpha}\in\mathcal{J}$, lines 10-12 of \textsc{ComputeIntegrals} requires $c_\facet=(2m(\facet)+4d)|\mathcal{J}|$ floating-point operations.

    It is not difficult to see that \textsc{ComputeIntegrals} is executed exactly once for each $k$-dimensional facet $\facet\subseteq\polytope$ with $0\le k\le d$. Thus, summing $c_\facet$ over each $\facet$, the time complexity given in the statement of the theorem is proven. The space complexity can be proven based on noting that $\ifun(\facet,\mathcal{J})$ is a set with $|\mathcal{J}|$ elements which must be stored for each facet $\facet\subseteq\polytope$.
\end{proof}

\begin{remark}[Simplifications for $\dim(\facet)=d$ and $\dim(\facet)=0$]\label{remark:monomial_flops_simplifications_0_d}
    When $\dim(\facet)=d$, any selection of $\x_\facet\in\reals^d$ can be made. By choosing $\x_\facet=\bm{0}$, the second sum in \reff{eqn:facet_to_subfacet_monomial} is eliminated and lines 10-12 in Algorithm \ref{alg:monomial_integration} can be performed in $(2m(\facet)+d)|\mathcal{J}|$ floating-point operations.

    When $\dim(\facet)=0$ (i.e. $\facet=\x_\facet\in\reals^d$), $\ifun(\facet,\bm{\alpha})$ can be performed in a couple of ways:
    \begin{itemize}
        \item Direct computation of $\ifun(\facet,\bm{\alpha})=\prod_{k=1}^d (\x_\facet)_k^{\alpha_k}$ - in this case, line 11 of Algorithm \ref{alg:monomial_integration} can be performed in $\bigO(d+\sum_{k=1}^d \log(1+\alpha_i))$ floating-point operations via binary exponentiation \cite{Knuth81}.
        \item Recursive computation of $\ifun(\facet,\bm{\alpha})=\alpha_k(\x_\facet)_k\ifun(\facet,\bm{\alpha}-\mathbf{e}_k)$ for some $1\le k\le d$ - in this case, line 11 of Algorithm \ref{alg:monomial_integration} can be performed in $\bigO(1)$ floating-point operations. 
    \end{itemize}
\end{remark}

The scalings of the time and space complexities reported in Theorem \ref{thm:monomial_integral_alg_complexity} as functions of the geometric complexity of $\polytope$ can be understood by visualising the recursive nature of Algorithm \ref{alg:monomial_integration} as a directed acyclic graph $G=G(\polytope)=(V,E)$. The vertex set $V=V(\polytope)$ is defined as the set of $k$-dimensional facets $\facet\subseteq\polytope$ for $0\le k\le d$. The edge set $E=E(\polytope)$ is defined as follows: for any facets $\facet_1, \facet_2\in V$, the directed edge $(\facet_1,\facet_2)\in E$ if and only if $\dim\facet_2 = \dim\facet_1-1$ and $\facet_2$ lies on the boundary of $\facet_1$. Figure~\ref{fig:tetrahedron_example} gives an example of the construction of $G(\polytope)$.

\begin{figure}[t!]
\centering
    \begin{subfigure}[B]{0.45\textwidth}
         \centering
         \resizebox{\textwidth}{!}{
         \begin{tikzpicture}
	       \node (c) at (0,0,0) {$c$};
	       \node (b) at (5,0,0) {$b$};
	       \node (d) at (0,5,0) {$d$};
	       \node (a) at (0,0,5) {$a$};
	
	       \draw[-] (a) -- (b);
    	   \draw[-] (a) -- (c);
	       \draw[-] (a) -- (d);
	       \draw[-] (b) -- (c);
	       \draw[-] (b) -- (d);
	       \draw[-] (c) -- (d);
        \end{tikzpicture}
        }
        \caption{A tetrahedron $T_3$ with labelled vertices.}
        \label{fig:tetrahedron}
     \end{subfigure}
     \hfill
     \begin{subfigure}[B]{0.45\textwidth}
         \centering
         \resizebox{\textwidth}{!}{
         \begin{tikzpicture} [roundnode/.style={circle, draw=black!60, fill=white!5, very thick, minimum size=7mm}]
        	\node[roundnode] (abcd) at (0,8) {$abcd$};
        	\node[roundnode] (abc) at (-3,6) {$abc$};
        	\node[roundnode] (abd) at (-1,6) {$abd$};
        	\node[roundnode] (acd) at (1,6) {$acd$};
        	\node[roundnode] (bcd) at (3,6) {$bcd$};
        	\node[roundnode] (ab) at (-4,4) {$ab$};
        	\node[roundnode] (ac) at (-2.4,4) {$ac$};
        	\node[roundnode] (ad) at (-0.8,4) {$ad$};
        	\node[roundnode] (bc) at (0.8,4) {$bc$};
        	\node[roundnode] (bd) at (2.4,4) {$bd$};
        	\node[roundnode] (cd) at (4,4) {$cd$};
        	\node[roundnode] (a) at (-3,2) {$a$};
        	\node[roundnode] (b) at (-1,2) {$b$};
        	\node[roundnode] (c) at (1,2) {$c$};
        	\node[roundnode] (d) at (3,2) {$d$};
        	\node[roundnode] (empty) at (0,0) {$\emptyset$};
        	
        	\draw[->] (abcd) -- (abc);
        	\draw[->] (abcd) -- (abd);
        	\draw[->] (abcd) -- (acd);
        	\draw[->] (abcd) -- (bcd);
        	
        	\draw[->] (abc) -- (ab);
        	\draw[->] (abc) -- (ac);
        	\draw[->] (abc) -- (bc);
        	\draw[->] (abd) -- (ab);
        	\draw[->] (abd) -- (ad);
        	\draw[->] (abd) -- (bd);
        	\draw[->] (acd) -- (ac);
        	\draw[->] (acd) -- (ad);
        	\draw[->] (acd) -- (cd);
        	\draw[->] (bcd) -- (bc);
        	\draw[->] (bcd) -- (bd);
        	\draw[->] (bcd) -- (cd);
        	
        	\draw[->] (ab) -- (a);
        	\draw[->] (ab) -- (b);
        	\draw[->] (ac) -- (a);
        	\draw[->] (ac) -- (c);
        	\draw[->] (ad) -- (a);
        	\draw[->] (ad) -- (d);
        	\draw[->] (bc) -- (b);
        	\draw[->] (bc) -- (c);
        	\draw[->] (bd) -- (b);
        	\draw[->] (bd) -- (d);
        	\draw[->] (cd) -- (c);
        	\draw[->] (cd) -- (d);
        	
        	\draw[dashed,->] (a) -- (empty);
        	\draw[dashed,->] (b) -- (empty);
        	\draw[dashed,->] (c) -- (empty);
        	\draw[dashed,->] (d) -- (empty);
        \end{tikzpicture}
        }
        \caption{The associated graph $G(T_3)$.}
        \label{fig:tetrahedron_face_lattice}
     \end{subfigure}
     \caption{Example of a tetrahedron $\polytope=T_3$ (left) and the associated recursive call graph $G(\polytope)$ (right). Each vertex of $G(\polytope)$ represents a facet of the tetrahedron (e.g. a vertex, an edge or a face). Edges between vertices in $G(\polytope)$  denote the relationship between facets on the boundaries of other facets (e.g. the edge $ab$ lies on the boundary of the face $abc$).}
     \label{fig:tetrahedron_example}
\end{figure}
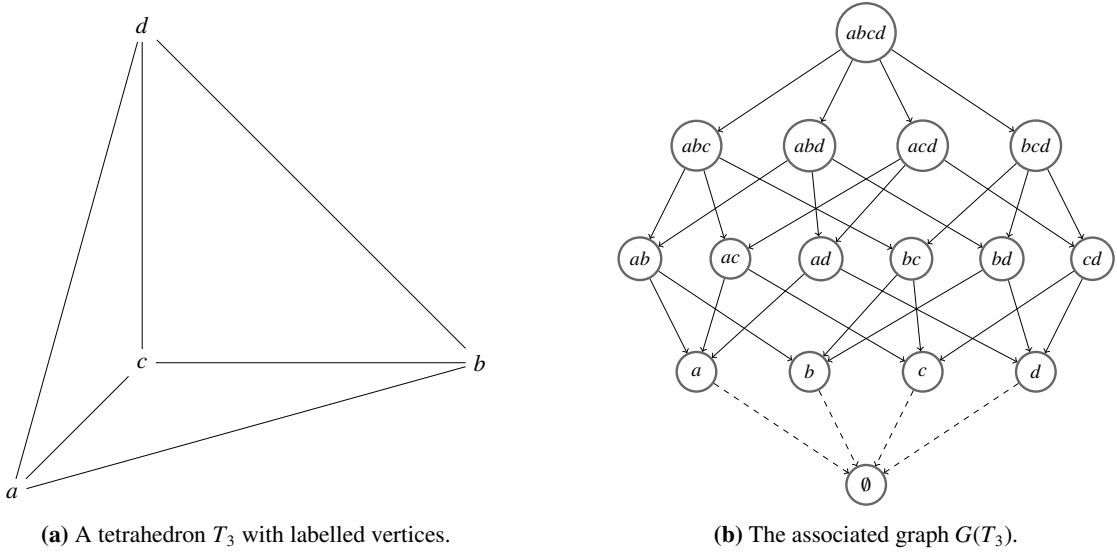

\begin{remark}[Facet lattice of $\polytope$] \label{remark:face_lattice_emptyset_inclusion}
    It is convenient to add an extra vertex $\emptyset$ to $V(\polytope)$, which we define as having dimension $-1$, and extra directed edges to $E(\polytope)$ from each vertex of $\polytope$ to $\emptyset$. The resulting graph $G(\polytope)$ then resembles a Hasse diagram representing the facet lattice formed from the facets of $\polytope$ ordered by inclusion \cite{bueler2000exact,kaibel2002computing}.
\end{remark}

\begin{remark}[Algorithm \ref{alg:monomial_integration} as a depth-first search] \label{remark:depth_first_search}
    The recursion of Algorithm 1 can be understood as a depth-first search of the graph $G(\polytope)$ starting at $\polytope$ where $\ifun(\facet,\mathcal{J})$ is assembled at each unvisited face $\facet$. By contrast, the implementation of Algorithm \ref{alg:monomial_integration} given in \cite{chin2020efficient} can be understood as a breadth-first search of the transpose graph $G'(\polytope)$ starting at $\emptyset$, where $G'(\polytope)$ differs from $G(\polytope)$ only by reversal of the directed edges.
\end{remark}

The time complexity of Algorithm \ref{alg:monomial_integration} can be understood in terms of the sizes of the vertex and edge sets $V(\polytope)$ and $E(\polytope)$. In particular, we have that
\begin{align*}
    \sum_{k=0}^d \sum_{\substack{\facet\subset\polytope \\ \dim\facet=k}} (2m(\facet)+4d) &= 2 \sum_{k=0}^d \sum_{\substack{\facet\subset\polytope \\ \dim\facet=k}} m(\facet) + 4d \sum_{k=0}^d \sum_{\substack{\facet\subset\polytope \\ \dim\facet=k}} 1 \\
    &= 2|E(\polytope)| + 4d|V(\polytope)|
\end{align*}
and the time and space complexities reported in Theorem \ref{thm:monomial_integral_alg_complexity} can be alternatively be written as \linebreak$\bigO\left( (|E(\polytope)|+d|V(\polytope)|) |\mathcal{J}|\right)$ and $\bigO(|V(\polytope)| |\mathcal{J}|)$, respectively. Table \ref{tab:example_time_space_complexities_monomial_integration} gives the time and space complexities of Algorithm \ref{alg:monomial_integration} for a number of different classes of polytopes, as well as the size of the vertex and edge sets $V(\polytope)$ and $E(\polytope)$ in the graph $G(\polytope)$, respectively.

\begin{table}[t!]
    \centering
    \begin{tabular}{c|cc|cc}
        Family & $|V(\polytope)|$ & $|E(\polytope)|$ & Time complexity & Space complexity \\ \hline
        $d$-dimensional simplex & $2^{d+1}$ & $2^d(d+1)$ & $\bigO(2^dd|\mathcal{J}|)$ & $\bigO(2^d|\mathcal{J}|)$ \\
        $d$-dimensional hypercube & $3^d+1$ & $2\cdot 3^{d-1}d$ & $\bigO(3^dd|\mathcal{J}|)$ & $\bigO(3^d|\mathcal{J}|)$ \\
        $n$-sided polygon & $2(n+1)$ & $4n$ & $\bigO(n|\mathcal{J}|)$ & $\bigO(n|\mathcal{J}|)$ \\
        $n$-gonal prism & $2(3n+2)$ & $15n+2$ & $\bigO(n|\mathcal{J}|)$ & $\bigO(n|\mathcal{J}|)$ \\
        $n$-based pyramid & $4(n+1)$ & $2(5n+1)$ & $\bigO(n|\mathcal{J}|)$ & $\bigO(n|\mathcal{J}|)$
    \end{tabular}
    \caption{Time and space complexities of Algorithm \ref{alg:monomial_integration} for different polytopes $\polytope$ as well as the set sizes $|V(\polytope)|$ and $|E(\polytope)|$ under the assumption that an extra vertex $\emptyset$ is added to $V(\polytope)$.}
    \label{tab:example_time_space_complexities_monomial_integration}
\end{table}

In practical finite element applications, $\polytope$ may denote a $d$-dimensional mesh element, $d=2,3$, in a polytopic mesh $\mathcal{T}$ of some domain of interest $\Omega$. Historically, simplicial or tensor-product elements have been used to partition the domain, though general polytopic elements have been proposed more recently, cf. \cite{Antonietti2016,doi:10.1142/S0218202512500492,cangiani2014hp,CaDoGeHo2017,hho_book_2021}, for example. The following theorem characterises the time and space complexities of Algorithm \ref{alg:monomial_integration} in these special cases.

\begin{theorem}[Complexity of Algorithm 1 for $d=2,3$] \label{thm:monomial_integral_alg_complexity_d_2_3}
    Let $\polytope\subset\reals^d$, $d=2,3$, denote a convex polygon or polyhedron.
    The time and space complexities of Algorithm~\ref{alg:monomial_integration}, measured in the sense described in Theorem~\ref{thm:monomial_integral_alg_complexity}, are $\bigO(e|\mathcal{J}|)$, where $e$ denotes the number of (1-dimensional) edges of $\polytope$.
\end{theorem}

\begin{proof}
    It suffices to show that $|V(\polytope)|=\bigO(e)$ and $|E(\polytope)|=\bigO(e)$.
    
    For the case $d=2$, a polygon $\polytope$ with $e$ edges also has $e$ vertices. Therefore, we have that $|V(\polytope)| = 2(e+1)$ and $|E(\polytope)|=4e$, as required.

    For the case $d=3$, let $v$ (respectively $f$) denote the number of vertices (respectively number of faces) of $\polytope$. Since $\polytope$ is convex, the Euler characteristic of the surface of $\polytope$ is equal to 2 \cite{grunbaum1967convex}, and hence
    \begin{equation*}
        v-e+f=2.
    \end{equation*}

    \noindent The size of the vertex set $V(\polytope)$ is given by
    \begin{align*}
        |V(\polytope)| = v+e+f+2 
        = 2(e+2).
    \end{align*}

    \noindent To compute the size of the edge set $E(\polytope)$, we note that each edge (or 1-dimensional facet) $\facet$ in $G(\polytope)$ has in-degree and out-degree 2, since each edge lies on the boundary of two faces and has two vertices on its own boundary. We therefore have that
    \begin{align*}
        |E(\polytope)| = v + 2e + 2e + f 
        = 5e + 2.
    \end{align*}
\end{proof}

\begin{remark}[Extension to non-convex polyhedra]\label{remark:non_convex_polyhedral_analysis_extension}
    The argument presented in the proof of Theorem \ref{thm:monomial_integral_alg_complexity_d_2_3} remains valid when both $\polytope$ and its $(d-1)$-dimensional facets $\{\facet_i\}_{i=1}^{m(\polytope)}$ are simply-connected. That is, Theorem \ref{thm:monomial_integral_alg_complexity_d_2_3} holds if neither $\polytope$ nor any of its facets have any holes.
\end{remark}

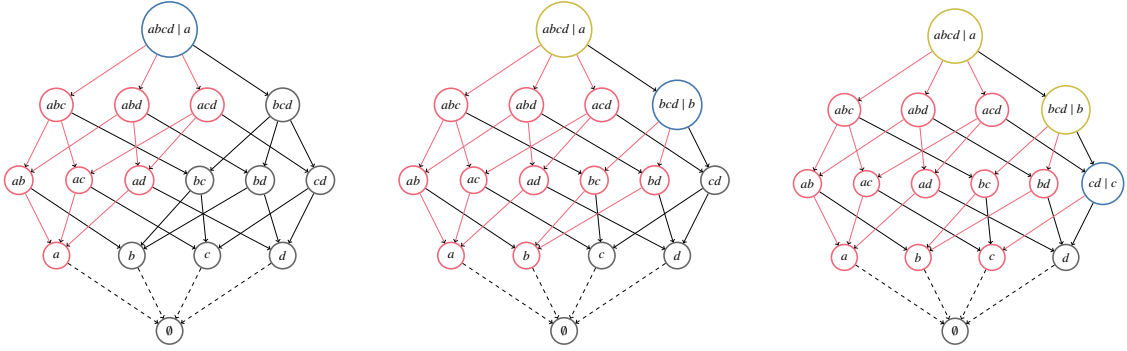
\begin{figure}[h]
    \centering
    \begin{subfigure}[B]{0.3\textwidth}
        \centering
        \resizebox{\textwidth}{!}{
        \begin{tikzpicture}[roundnode/.style={circle, draw=black!60, fill=white!5, very thick, minimum size=7mm}]
        	\node[roundnode,draw=tolcol1] (abcd) at (0,8) {$abcd \mid a$};
        	\node[roundnode,draw=tolcol2] (abc) at (-3,6) {$abc$};
        	\node[roundnode,draw=tolcol2] (abd) at (-1,6) {$abd$};
        	\node[roundnode,draw=tolcol2] (acd) at (1,6) {$acd$};
        	\node[roundnode] (bcd) at (3,6) {$bcd$};
        	\node[roundnode,draw=tolcol2] (ab) at (-4,4) {$ab$};
        	\node[roundnode,draw=tolcol2] (ac) at (-2.4,4) {$ac$};
        	\node[roundnode,draw=tolcol2] (ad) at (-0.8,4) {$ad$};
        	\node[roundnode] (bc) at (0.8,4) {$bc$};
        	\node[roundnode] (bd) at (2.4,4) {$bd$};
        	\node[roundnode] (cd) at (4,4) {$cd$};
        	\node[roundnode,draw=tolcol2] (a) at (-3,2) {$a$};
        	\node[roundnode] (b) at (-1,2) {$b$};
        	\node[roundnode] (c) at (1,2) {$c$};
        	\node[roundnode] (d) at (3,2) {$d$};
        	\node[roundnode] (empty) at (0,0) {$\emptyset$};
        	
        	\draw[->,draw=tolcol2] (abcd) -- (abc);
        	\draw[->,draw=tolcol2] (abcd) -- (abd);
        	\draw[->,draw=tolcol2] (abcd) -- (acd);
        	\draw[->] (abcd) -- (bcd);
        	
        	\draw[->,draw=tolcol2] (abc) -- (ab);
        	\draw[->,draw=tolcol2] (abc) -- (ac);
        	\draw[->] (abc) -- (bc);
        	\draw[->,draw=tolcol2] (abd) -- (ab);
        	\draw[->,draw=tolcol2] (abd) -- (ad);
        	\draw[->] (abd) -- (bd);
        	\draw[->,draw=tolcol2] (acd) -- (ac);
        	\draw[->,draw=tolcol2] (acd) -- (ad);
        	\draw[->] (acd) -- (cd);
        	\draw[->] (bcd) -- (bc);
        	\draw[->] (bcd) -- (bd);
        	\draw[->] (bcd) -- (cd);
        	
        	\draw[->,draw=tolcol2] (ab) -- (a);
        	\draw[->] (ab) -- (b);
        	\draw[->,draw=tolcol2] (ac) -- (a);
        	\draw[->] (ac) -- (c);
        	\draw[->,draw=tolcol2] (ad) -- (a);
        	\draw[->] (ad) -- (d);
        	\draw[->] (bc) -- (b);
        	\draw[->] (bc) -- (c);
        	\draw[->] (bd) -- (b);
        	\draw[->] (bd) -- (d);
        	\draw[->] (cd) -- (c);
        	\draw[->] (cd) -- (d);
        	
        	\draw[dashed,->] (a) -- (empty);
        	\draw[dashed,->] (b) -- (empty);
        	\draw[dashed,->] (c) -- (empty);
        	\draw[dashed,->] (d) -- (empty);
        \end{tikzpicture}
        }
    \end{subfigure}
    \hfill
    \begin{subfigure}[B]{0.3\textwidth}
        \centering
        \resizebox{\textwidth}{!}{
        \begin{tikzpicture}[roundnode/.style={circle, draw=black!60, fill=white!5, very thick, minimum size=7mm}]
        	\node[roundnode,draw=tolcol3] (abcd) at (0,8) {$abcd \mid a$};
        	\node[roundnode,draw=tolcol2] (abc) at (-3,6) {$abc$};
        	\node[roundnode,draw=tolcol2] (abd) at (-1,6) {$abd$};
        	\node[roundnode,draw=tolcol2] (acd) at (1,6) {$acd$};
        	\node[roundnode,draw=tolcol1] (bcd) at (3,6) {$bcd \mid b$};
        	\node[roundnode,draw=tolcol2] (ab) at (-4,4) {$ab$};
        	\node[roundnode,draw=tolcol2] (ac) at (-2.4,4) {$ac$};
        	\node[roundnode,draw=tolcol2] (ad) at (-0.8,4) {$ad$};
        	\node[roundnode,draw=tolcol2] (bc) at (0.8,4) {$bc$};
        	\node[roundnode,draw=tolcol2] (bd) at (2.4,4) {$bd$};
        	\node[roundnode] (cd) at (4,4) {$cd$};
        	\node[roundnode,draw=tolcol2] (a) at (-3,2) {$a$};
        	\node[roundnode,draw=tolcol2] (b) at (-1,2) {$b$};
        	\node[roundnode] (c) at (1,2) {$c$};
        	\node[roundnode] (d) at (3,2) {$d$};
        	\node[roundnode] (empty) at (0,0) {$\emptyset$};
        	
        	\draw[->,draw=tolcol2] (abcd) -- (abc);
        	\draw[->,draw=tolcol2] (abcd) -- (abd);
        	\draw[->,draw=tolcol2] (abcd) -- (acd);
        	\draw[->] (abcd) -- (bcd);
        	
        	\draw[->,draw=tolcol2] (abc) -- (ab);
        	\draw[->,draw=tolcol2] (abc) -- (ac);
        	\draw[->] (abc) -- (bc);
        	\draw[->,draw=tolcol2] (abd) -- (ab);
        	\draw[->,draw=tolcol2] (abd) -- (ad);
        	\draw[->] (abd) -- (bd);
        	\draw[->,draw=tolcol2] (acd) -- (ac);
        	\draw[->,draw=tolcol2] (acd) -- (ad);
        	\draw[->] (acd) -- (cd);
        	\draw[->,draw=tolcol2] (bcd) -- (bc);
        	\draw[->,draw=tolcol2] (bcd) -- (bd);
        	\draw[->] (bcd) -- (cd);
        	
        	\draw[->,draw=tolcol2] (ab) -- (a);
        	\draw[->] (ab) -- (b);
        	\draw[->,draw=tolcol2] (ac) -- (a);
        	\draw[->] (ac) -- (c);
        	\draw[->,draw=tolcol2] (ad) -- (a);
        	\draw[->] (ad) -- (d);
        	\draw[->,draw=tolcol2] (bc) -- (b);
        	\draw[->] (bc) -- (c);
        	\draw[->,draw=tolcol2] (bd) -- (b);
        	\draw[->] (bd) -- (d);
        	\draw[->] (cd) -- (c);
        	\draw[->] (cd) -- (d);
        	
        	\draw[dashed,->] (a) -- (empty);
        	\draw[dashed,->] (b) -- (empty);
        	\draw[dashed,->] (c) -- (empty);
        	\draw[dashed,->] (d) -- (empty);
        \end{tikzpicture}
        }
    \end{subfigure}
    \hfill
    \begin{subfigure}[B]{0.3\textwidth}
        \centering
        \resizebox{\textwidth}{!}{
        \begin{tikzpicture}[roundnode/.style={circle, draw=black!60, fill=white!5, very thick, minimum size=7mm}]
        	\node[roundnode,draw=tolcol3] (abcd) at (0,8) {$abcd \mid a$};
        	\node[roundnode,draw=tolcol2] (abc) at (-3,6) {$abc$};
        	\node[roundnode,draw=tolcol2] (abd) at (-1,6) {$abd$};
        	\node[roundnode,draw=tolcol2] (acd) at (1,6) {$acd$};
        	\node[roundnode,draw=tolcol3] (bcd) at (3,6) {$bcd \mid b$};
        	\node[roundnode,draw=tolcol2] (ab) at (-4,4) {$ab$};
        	\node[roundnode,draw=tolcol2] (ac) at (-2.4,4) {$ac$};
        	\node[roundnode,draw=tolcol2] (ad) at (-0.8,4) {$ad$};
        	\node[roundnode,draw=tolcol2] (bc) at (0.8,4) {$bc$};
        	\node[roundnode,draw=tolcol2] (bd) at (2.4,4) {$bd$};
        	\node[roundnode,draw=tolcol1] (cd) at (4,4) {$cd \mid c$};
        	\node[roundnode,draw=tolcol2] (a) at (-3,2) {$a$};
        	\node[roundnode,draw=tolcol2] (b) at (-1,2) {$b$};
        	\node[roundnode,draw=tolcol2] (c) at (1,2) {$c$};
        	\node[roundnode] (d) at (3,2) {$d$};
        	\node[roundnode] (empty) at (0,0) {$\emptyset$};
        	
        	\draw[->,draw=tolcol2] (abcd) -- (abc);
        	\draw[->,draw=tolcol2] (abcd) -- (abd);
        	\draw[->,draw=tolcol2] (abcd) -- (acd);
        	\draw[->] (abcd) -- (bcd);
        	
        	\draw[->,draw=tolcol2] (abc) -- (ab);
        	\draw[->,draw=tolcol2] (abc) -- (ac);
        	\draw[->] (abc) -- (bc);
        	\draw[->,draw=tolcol2] (abd) -- (ab);
        	\draw[->,draw=tolcol2] (abd) -- (ad);
        	\draw[->] (abd) -- (bd);
        	\draw[->,draw=tolcol2] (acd) -- (ac);
        	\draw[->,draw=tolcol2] (acd) -- (ad);
        	\draw[->] (acd) -- (cd);
        	\draw[->,draw=tolcol2] (bcd) -- (bc);
        	\draw[->,draw=tolcol2] (bcd) -- (bd);
        	\draw[->] (bcd) -- (cd);
        	
        	\draw[->,draw=tolcol2] (ab) -- (a);
        	\draw[->] (ab) -- (b);
        	\draw[->,draw=tolcol2] (ac) -- (a);
        	\draw[->] (ac) -- (c);
        	\draw[->,draw=tolcol2] (ad) -- (a);
        	\draw[->] (ad) -- (d);
        	\draw[->,draw=tolcol2] (bc) -- (b);
        	\draw[->] (bc) -- (c);
        	\draw[->,draw=tolcol2] (bd) -- (b);
        	\draw[->] (bd) -- (d);
        	\draw[->,draw=tolcol2] (cd) -- (c);
        	\draw[->] (cd) -- (d);
        	
        	\draw[dashed,->] (a) -- (empty);
        	\draw[dashed,->] (b) -- (empty);
        	\draw[dashed,->] (c) -- (empty);
        	\draw[dashed,->] (d) -- (empty);
        \end{tikzpicture}
        }
    \end{subfigure}
    \caption{Effect of pruning on the recursive call graph for Algorithm \ref{alg:monomial_integration} in the case $\polytope=T_3$ as in Figure \ref{fig:tetrahedron_example}. Left-to-right: first three recursive executions of \textsc{ComputeIntegrals}. Blue node: facet $\facet$ associated with current execution of $\textsc{ComputeIntegrals}(\facet,\mathcal{J})$ and choice of reference point $\x_\facet$. Yellow nodes: facets $\facet$ associated with previous executions of $\textsc{ComputeIntegrals}(\facet,\mathcal{J})$ and choice of reference point $\x_\facet$. Red nodes: facets $\facet$ eliminated from recursion as a result of pruning; i.e., unvisited facets with $\dist(\facet,\x_{\facet'})=0$ for some previously-selected reference point $\x_{\facet'}$.}
    \label{fig:tetrahedron_face_lattice_pruned}
 \end{figure}

One may reduce the time and space complexities of Algorithm \ref{alg:monomial_integration} through judicious selection of the reference points $\x_\facet$. When $\x_\facet$ is chosen as a vertex of $\facet$, one may avoid a number of recursive calls to \textsc{ComputeIntegrals} for some boundary facets of $\facet$. We shall refer to the resulting implementation as \emph{pruned} - this is illustrated in Figure \ref{fig:tetrahedron_face_lattice_pruned}. While a complete time and space complexity analysis of Algorithm \ref{alg:monomial_integration} with pruning is not presented here, pruning can lead to significant computational savings on simple geometries. For instance, the time complexity of Algorithm \ref{alg:monomial_integration} in the case where $\polytope$ is a $d$-dimensional simplex is $\bigO(2^dd|\mathcal{J}|)$; with pruning, this can be reduced to $\bigO(d^2|\mathcal{J}|)$. On more complicated domains, the effects of pruning are likely to be less significant. To our knowledge, finding an optimal pruning of $G(\polytope)$ for a general polytope - that is, a shortest sequence of visited facets $(\facet_n)_{n\ge0}$ and corresponding reference points $(\x_{\facet_n})_{n\ge0}$ for which Algorithm \ref{alg:monomial_integration} can compute $\ifun(\polytope,\mathcal{J})$ - is an open problem.

\subsubsection{Time complexity of distance pre-computation}

It is convenient to omit the computation of distances of the form $\dist(\partial\facet_i,\x_\facet)$ in the proof of Theorem \ref{thm:monomial_integral_alg_complexity} since such quantities do not need to be re-computed for each $\ifun(\facet,\bm{\alpha})$. In cases where $|\mathcal{J}|$ is very small or $\polytope$ is high-dimensional, the evaluation of these distances can become the most computationally-expensive part of Algorithm \ref{alg:monomial_integration}. For instance, it is shown in \cite{bueler2000exact} that the time complexity of the computation of the volume of a $d$-dimensional hypercube via the quadrature-free integration method is $\bigO(d^43^d)$, which is a factor of $\bigO(d^3)$ larger than that predicted by Theorem \ref{thm:monomial_integral_alg_complexity}. To remedy this, the following lemma measures the complexity of operations omitted in the proof of Theorem \ref{thm:monomial_integral_alg_complexity}.
\begin{lemma}[Time complexity of distance pre-computation]
    The time complexity of evaluating the set
    \begin{equation*}
        \left\{ \dist(\partial\facet_i,\x_\facet) : 1\le i\le m(\facet), 0\le\dim\facet\le d \right\}
    \end{equation*}

    \noindent is $\bigO(\chi_3(\polytope))$, where
    \begin{equation*}
        \chi_3(\polytope) = d \sum_{k=0}^{d} \sum_{\substack{\facet\subseteq\polytope \\ \dim\facet=k}} k^2 m(\facet).
    \end{equation*}

    \noindent Alternatively, the time complexity may be expressed as $\bigO(d^3|E(\polytope)|)$.
\end{lemma}

\begin{proof}
    Let $\facet$ denote a $k$-dimensional facet with selected reference point $\x_\facet$. Let $\partial\facet$ denote any of its $(k-1)$-dimensional boundary facets. Suppose that $\partial\facet$ has $n$ vertices $\{\x_i\}_{i=0}^{n-1}$ and note that $k\le n$. The $(k-1)$-dimensional hyperplane $\mathcal{H}$ containing $\partial\facet$ can be uniquely defined by the points $\{\x_i\}_{i=0}^{k-1}$: for any $\x\in\mathcal{H}$, there exists $\mathbf{t}=(t_i)_{i=1}^{k-1}\in\reals^{k-1}$ such that
    \begin{equation*}
        \x = \x(\mathbf{t}) = \x_0 + \sum_{i=1}^{k-1} (\x_i-\x_0) t_i.
    \end{equation*}

    \noindent The distance $\dist(\partial\facet,\x_\facet)$ (or, more precisely, the distance between $\mathcal{H}$ and $\x_\facet$) is the minimum value of $|| \x(\mathbf{t})-\x_\facet ||_2$ over $\mathbf{t}\in\reals^{k-1}$; this occurs when $\mathbf{t}$ solves
    \begin{equation*}
        \mathbf{At} = \mathbf{f},
    \end{equation*}

    \noindent where the entries of $\mathbf{A}\in\reals^{(k-1)\times(k-1)}$ and $\mathbf{f}\in\reals^{k-1}$ are given by $(\mathbf{A})_{ij} = (\x_i-\x_0)\cdot(\x_j-\x_0)$ and $(\mathbf{f})_i = (\x_0-\x_\facet)\cdot(\x_i-\x_0)$, respectively. The number of floating-point operations required to assemble and solve the linear system above is $\bigO(k^2d)$ and $\bigO(k^3)$, respectively; thus, the time complexity of computing $\dist(\partial\facet,\x_\facet)$ for a single $(k-1)$-dimensional facet $\partial\facet$ is $\bigO(k^2d)$.

    For each $k$-dimensional facet $\facet$, $0\le k\le d$, $\dist(\partial\facet_i,\x_\facet)$ is computed for each $1\le i\le m(\facet)$. Summing the time complexity of a single computation of $\dist(\partial\facet,\x_\facet)$ over these limits yields the first result stated in the proof; the second result is obtained by bounding the complexity of a single computation of $\dist(\partial\facet,\x_\facet)$ from above by $\bigO(d^3)$.
\end{proof}

\begin{corollary}[Time complexity of Algorithm \ref{alg:monomial_integration} with distance computation] \label{coroll:monom_int_final_time_complexity}
    The time complexity of Algorithm \ref{alg:monomial_integration} including the evaluation of the set
    \begin{equation*}
        \left\{ \dist(\partial\facet_i,\x_\facet) : 1\le i\le m(\facet), 0\le\dim\facet\le d \right\}
    \end{equation*}

    \noindent is $\bigO\left(d^3|E(\polytope)| + \left(d|V(\polytope)|+|E(\polytope)|\right) |\mathcal{J}|\right)$.
\end{corollary}

\section{DGFEM discretisation of the transport equation} \label{sec:dgfem}

In this section we consider the application of the numerical integration algorithm outlined in the previous section for the computation of the volume integrals arising within the DGFEM discretisation of the linear transport equation. To this end, given an open bounded polyhedral domain $\Omega\subset\reals^d$, $d=2,3$, we consider the following advection-reaction equation: find $u:\Omega\rightarrow\reals$ such that
\begin{align}
    \nabla\cdot(\mathbf{b} u) + cu &= f \quad\textnormal{ in }\Omega, \label{eqn:linear_transport_pde} \\
    u &= g \quad\textnormal{ on } \Gamma_{in}, \nonumber
\end{align}

\noindent where $c,f:\Omega\rightarrow\reals$, $g:\Gamma_{in}\rightarrow\reals$ and $\mathbf{b}:\Omega\rightarrow\reals^d$ are given data terms, and $\Gamma_{in}=\{ \x\in\partial\Omega : \mathbf{b}\cdot\mathbf{n}(\x)<0 \}$ denotes the inflow boundary of $\Omega$, where $\mathbf{n}(\x)$ denotes the outward unit normal to $\Omega$ at $\x\in\partial\Omega$. In the following, we assume that $\mathbf{b}$ is a constant velocity vector, while $c$ is a given as a piecewise-constant function with respect to the elements in the underlying finite element mesh $\mathcal{T}$ defined below. For the type of applications we have in mind, namely the numerical approximation of the linear Boltzmann transport problem, cf.~\cite{houston2023efficient}, these assumptions are not restrictive. That said, the proposed quadrature free implementation can easily be extended to include the case when $\mathbf{b}$, $c$, $f$ and $g$ are polynomial functions (or piecewise-polynomial with respect to the elements of $\mathcal{T})$.

\subsection{Discretisation}

We discretise the linear transport equation \reff{eqn:linear_transport_pde} using a DGFEM approach. To this end, let $\mathcal{T}$ denote a subdivision of the spatial domain $\Omega$ into open non-overlapping polytopic elements $\kappa$ such that $\bar{\Omega}=\bigcup_{\kappa\in\mathcal{T}} \bar{\kappa}$. We denote by $\mathcal{E}$ the set of faces in $\mathcal{T}$, which are defined as the $(d-1)$-dimensional planar facets of elements $\kappa\in\mathcal{T}$.

To each $\kappa\in\mathcal{T}$ we respectively denote by $h_\kappa>0$ and $p_\kappa\ge0$ the diameter of $\kappa$ and the polynomial degree on $\kappa$. The spatial finite element space is defined by
\begin{equation*}
    \mathbb{V} = \left\{ v\in L_2(\Omega) : v|_\kappa\in\mathbb{P}^{p_\kappa}(\kappa) \textnormal{ for all }\kappa\in\mathcal{T} \right\},
\end{equation*}

\noindent where $\mathbb{P}^{p_\kappa}(\kappa)$ denotes the space of all $d$-variate polynomials with maximum total degree at most $p_\kappa$ on $\kappa$.

Given $\kappa\in\mathcal{T}$, we define the inflow and outflow parts of $\partial\kappa$ by
\begin{align*}
    \partial_-\kappa &= \left\{ \x\in\partial\kappa : \mathbf{b}\cdot\mathbf{n}(\x) < 0 \right\}, \\
    \partial_+\kappa &= \left\{ \x\in\partial\kappa : \mathbf{b}\cdot\mathbf{n}(\x) \ge 0 \right\},
\end{align*}

\noindent respectively, where $\mathbf{n}(\x)$ denotes the outward unit normal to $\kappa$ at $\x\in\partial\kappa$. For a sufficiently-smooth function $v$, we denote by $v_\kappa^+$ (respectively, $v_\kappa^-$) the interior (respectively, exterior) trace of $v$ on $\partial\kappa$ (respectively, $\partial\kappa\setminus\partial\Omega$). Since it will always be clear which element $\kappa\in\mathcal{T}$ the quantities $v_\kappa^\pm$ correspond to, the subscript $\kappa$ will be suppressed for the remainder of this article.

The DGFEM discretisation of \reff{eqn:linear_transport_pde} with upwind numerical flux reads as follows: find $u_h\in\mathbb{V}$ such that
\begin{equation} \label{eqn:dgfem_linear_transport}
    a(u_h,v_h) = \ell(v_h)
\end{equation}

\noindent for all $v_h\in\mathbb{V}$, where $a:\mathbb{V}\times\mathbb{V}\rightarrow\reals$ and $\ell:\mathbb{V}\rightarrow\reals$ are defined, respectively, for all $w_h,v_h\in\mathbb{V}$ by
\begin{align*}
    a(w_h,v_h) &= \sum_{\kappa\in\mathcal{T}} \Bigg( \int_\kappa \left( -w_h\mathbf{b}\cdot\nabla v_h + c w_hv_h \right) \d\x + \int_{\partial_+\kappa} |\mathbf{b}\cdot\mathbf{n}| \ w_h^+ v_h^+ \d s \\
    &\quad\quad\quad\quad - \int_{\partial_-\kappa\setminus\partial\Omega} |\mathbf{b}\cdot\mathbf{n}| \ w_h^-v_h^+ \d s \Bigg), \\
    \ell(v_h) &= \sum_{\kappa\in\mathcal{T}} \left( \int_\kappa f v_h \d\x + \int_{\partial_-\kappa\cap\partial\Omega} |\mathbf{b}\cdot\mathbf{n}| \ g v_h^+ \d s \right).
\end{align*}

\subsection{Basis functions on polytopic elements}

For the remainder of this article we will assume that each element $\kappa\in\mathcal{T}$ is equipped with a basis comprising of the polynomial space $\mathbb{P}^p(\kappa)$ for some fixed $p\ge0$. Following \cite{antonietti2018fast,cangiani2014hp}, we construct a Cartesian bounding box $B_\kappa$ for each element $\kappa\in\mathcal{T}$ such that $\bar{\kappa}\subseteq\bar{B_\kappa}$. Furthermore, we define a reference bounding box $\hat{B}=(-1,1)^d$ and a family of affine mappings $F_\kappa:\hat{B}\rightarrow B_\kappa$ such that $F_\kappa(\hat{\x})=\mathbf{J}_\kappa\hat{\x}+\mathbf{t}_\kappa$, where $\mathbf{J}_\kappa\in\reals^{d\times d}$ is the (diagonal) Jacobi matrix of $F_\kappa$ and $\mathbf{t}_\kappa\in\reals^d$ is the translation between the origin $\bm{0}\in\hat{B}$ and the barycentre of $B_\kappa$.

We define a basis of $\mathbb{P}^p(\hat{B})$ as follows: we denote by $\{\mathcal{L}_n(t)\}_{n=0}^\infty$ the family of orthogonal (or orthonormal) univariate Legendre polynomials on $L_2(-1,1)$. For each multi-index $\bm{\alpha}$ of length $d$ with $0\le|\bm{\alpha}|\le p$ we define the basis function $\hat{\phi}_{\bm{\alpha}}:\hat{B}\rightarrow\reals$ by
\begin{equation} \label{eqn:basis_fn_ref_ele_defn}
    \hat{\phi}_{\bm{\alpha}}(\hat{\x}) = \prod_{k=1}^d \mathcal{L}_{\alpha_k}(\hat{x}_k).
\end{equation}

\noindent It is straightforward to see that $\{\hat{\phi}_{\bm{\alpha}}(\hat{\x})\}_{0\le|\bm{\alpha}|\le p}$ forms a basis for $\mathbb{P}^p(\hat{B})$. The basis functions $\{\phi_{\bm{\alpha},\kappa}(\x)\}_{0\le|\bm{\alpha}|\le p}$ for $\mathbb{P}^p(\kappa)$ are constructed upon application of the element mapping; more precisely, $\phi_{\bm{\alpha},\kappa}(\x) = \hat{\phi}_{\bm{\alpha}}(F_\kappa^{-1}(\x))$. The set
\begin{equation*}
    \{ \phi_{\bm{\alpha},\kappa}(\x) : \kappa\in\mathcal{T}, 0\le|\bm{\alpha}|\le p \}
\end{equation*}

\noindent forms a basis on each element $\kappa$, $\kappa\in\mathcal{T}$, for the finite element space $\mathbb{V}$. Henceforth, we will identify a bijection between the set of multi-indices $\{\bm{\alpha}\}_{0\le|\bm{\alpha}|\le p}$ and the set $\{1,\dots,\dim\mathbb{P}^p(\kappa)\}$ such that the $i^{th}$ basis function on $\kappa$ is denoted by $\phi_{\bm{\alpha}^{(i)},\kappa}(\x)$.

We conclude this section by expanding the products of $\mathcal{L}_n(t)$ and their derivatives as a sum of monomials:
\begin{align}
    \mathcal{L}_m(t) \mathcal{L}_n(t) &= \sum_{k=0}^{m+n} C_{m,n,k} t^k, \label{eqn:decomp_prod_leg_poly} \\
    \mathcal{L}_m'(t) \mathcal{L}_n(t) &= \sum_{k=0}^{m+n-1} C_{m,n,k}' t^k. \label{eqn:decomp_prod_leg_poly_deriv} 
\end{align}

\noindent The sets of coefficients $\{ C_{m,n,k} : 0\le m,n\le p, 0\le k\le m+n \}$ and $\{ C_{m,n,k}' : 0\le m,n\le p, 0\le k\le m+n-1 \}$ may be pre-computed before assembly.

\section{Analysis of volume matrix assembly} \label{sec:matrix_assembly}

In this section we focus on the efficient assembly of the DGFEM matrix arising on the right-hand side of \reff{eqn:dgfem_linear_transport}. Typically, when quadrature is employed, the evaluation of the volume integrals appearing in the right-hand side of \reff{eqn:dgfem_linear_transport} are far more expensive to compute than the corresponding face integrals since in the former case significantly more quadrature points must be employed. 
With that in mind, we focus on the acceleration of the assembly of the local element matrix contributions $\mathbf{A}_\kappa\in\reals^{\dim\mathbb{P}^p(\kappa)\times\dim\mathbb{P}^{p}(\kappa)}$, where
\begin{equation} \label{eqn:local_matrix_component_wise_defn}
    (\mathbf{A}_\kappa)_{ij} = \int_\kappa \left( -\phi_{\bm{\alpha}^{(j)},\kappa}(\x) \mathbf{b}\cdot\nabla \phi_{\bm{\alpha}^{(i)},\kappa}(\x) + c \phi_{\bm{\alpha}^{(i)},\kappa}(\x) \phi_{\bm{\alpha}^{(j)},\kappa}(\x) \right) \d\x.
\end{equation}
While we will not discuss the exploitation of quadrature-free methods to compute the face integrals appearing in \reff{eqn:dgfem_linear_transport}, we refer to \cite{antonietti2018fast} for the application of the proposed numerical integration approach for the computation of the face integrals arising from a DGFEM discretisation of a second-order elliptic problem.

\subsection{Quadrature based assembly}

As a point of comparison, we will consider quadrature rules of the form
\begin{equation} \label{eqn:quadrature_approximation}
    \int_\polytope f(\x) \d\x \approx \sum_{i=1}^N \omega_i f(\x_i),
\end{equation}

\noindent where $f:\polytope\rightarrow\reals$, $\{\x_i\}_{i=1}^N\subset\reals^d$ denotes a set of quadrature points with non-negative quadrature weights $\{\omega_i\}_{i=1}^N\subset\reals$.

A number of methods may be employed to construct the $N$-point quadrature scheme $Q_N=\{(\x_i,\omega_i)\}_{i=1}^N$. 
One of the most popular and simple to implement strategies is to sub-tessellate the integration domain $\polytope$ into simplicial subdomains (triangles in 2D, tetrahedra in 3D), on which standard quadrature schemes can be used \cite{sukumar2004conforming}. However, the resulting quadrature scheme on $\polytope$ may contain an excessive number of points and weights. It has been demonstrated that numerical optimisation algorithms can generate efficient numerical quadrature schemes on arbitrary polygonal domains, cf., for example, \cite{mousavi2010generalized}.

In the case when $f\in\mathbb{P}^p(\polytope)$, the approximation \reff{eqn:quadrature_approximation} can be exact. Indeed, it can be shown that the smallest $N$-point quadrature scheme which is exact for all polynomial functions of a given degree $p$ contains
\begin{equation} \label{eqn:quad_point_size_lower_bnd}
    N\ge \binom{\lfloor\frac{p}{2}\rfloor+d}{d}
\end{equation}

\noindent quadrature points and weights \cite{stroud1971approximate}.

\begin{algorithm}[t!]
\caption{Computation of $\mathbf{A}_\kappa$ using quadrature.}
\label{alg:stiffness_assembly_quadrature}
\begin{algorithmic}[1]
    \Procedure{ComputeElementMatrix}{$\kappa$}
        \State Compute quadrature scheme $\{(\x_q,\omega_q)\}_{q=1}^N$ on $\kappa$
        \State Pre-compute $\{\phi_{\bm{\alpha},\kappa}\}_{0\le|\bm{\alpha}|\le p}$ and $\{\nabla\phi_{\bm{\alpha},\kappa}\}_{0\le|\bm{\alpha}|\le p}$ at quadrature points
        \State $\mathbf{A}_\kappa\gets0$
        \For{$q=1,\dots,N$}
            \For{$i=1,\dots,\dim\mathbb{P}^p(\kappa)$}
                \For{$j=1,\dots,\dim\mathbb{P}^p(\kappa)$}
                    \State $I\gets \left( c\phi_{\bm{\alpha}^{(i)},\kappa}(\x_q) - \mathbf{b}\cdot\nabla\phi_{\bm{\alpha}^{(i)},\kappa}(\x_q) \right) \phi_{\bm{\alpha}^{(j)},\kappa}(\x_q)$
                    \State $(\mathbf{A}_\kappa)_{i,j} \gets (\mathbf{A}_\kappa)_{i,j} + \omega_q I$
                \EndFor
            \EndFor
        \EndFor
        \State \Return $\mathbf{A}_\kappa$
    \EndProcedure
\end{algorithmic}
\end{algorithm}

Algorithm \ref{alg:stiffness_assembly_quadrature} presents a pseudocode for a typical implementation of numerical quadrature to evaluate the integrals \reff{eqn:local_matrix_component_wise_defn} appearing in the DGFEM discretisation of the transport equation. Noting that lines 8 and 9 of Algorithm 2 require $2(d+1)$ and $2$ floating-point operations to evaluate, respectively, it can be seen that the number of floating-point operations performed in the main body (lines 5-12) is given by $2(d+2) (\dim\mathbb{P}^p(\kappa))^2 N$. Furthermore, noting that $\dim\mathbb{P}^p(\kappa)=\binom{p+d}{d}$ and that the integrand of \reff{eqn:local_matrix_component_wise_defn} is a polynomial of total degree at most $2p$, the number of floating-point operations required to exactly evaluate $\mathbf{A}_\kappa$ for a single element $\kappa$ using Algorithm \ref{alg:stiffness_assembly_quadrature} is at least $2(d+1)\binom{p+d}{d}^3$. Here, we have assumed that the lower bound on the number of quadrature points in \reff{eqn:quad_point_size_lower_bnd} is attainable; that is, we set $N=N_{opt}=\binom{p+d}{d}$.

\subsection{Quadrature-free based assembly}

The quadrature-free integration method outlined in Section \ref{section:monomial_integration_quad_free} is not immediately applicable to the case of exactly evaluating the entries of $\mathbf{A}_\kappa$ in \reff{eqn:local_matrix_component_wise_defn}, since the integrand is typically not a homogeneous function. We remedy this issue by decomposing the integrand as a sum of monomials which may be integrated separately using Algorithm \ref{alg:monomial_integration}. We shall skip the details for brevity; see \cite{antonietti2018fast} for a more detailed treatment of similar integrals.

Since the basis functions $\{\phi_{\bm{\alpha},\kappa}\}_{0\le|\bm{\alpha}|\le p}$ are only supported on $\kappa$, we may apply the inverse map $F_\kappa^{-1}$ to obtain an expression for the matrix entry $(\mathbf{A}_{\kappa})_{i,j}$ as an integral over the mapped element $\hat{\kappa}=F_\kappa^{-1} (\kappa) \subseteq \hat{B}$. It can be shown that
\begin{align*}
    (\mathbf{A}_{\kappa})_{i,j} &= \int_{\hat{\kappa}} \left( -\hat{\phi}_{\bm{\alpha}^{(j)}}(\hat{\x}) \hat{\mathbf{b}}_{\kappa} \cdot \hat{\nabla} \hat{\phi}_{\bm{\alpha}^{(i)}}(\hat{\x}) + c \hat{\phi}_{\bm{\alpha}^{(i)}}(\hat{\x}) \hat{\phi}_{\bm{\alpha}^{(j)}}(\hat{\x}) \right) |\mathbf{J}_\kappa| \d\hat{\x},
\end{align*}

\noindent where $\hat{\mathbf{b}}_\kappa=\mathbf{J}_\kappa^{-1}\mathbf{b}$ denotes a scaled wind direction. Finally, by using the definition \reff{eqn:basis_fn_ref_ele_defn} of the basis functions $\{\hat{\phi}_{\bm{\alpha}}\}_{0\le\bm{\alpha}\le p}$ (which we remark are independent of $\kappa$) and the decompositions \reff{eqn:decomp_prod_leg_poly} and \reff{eqn:decomp_prod_leg_poly_deriv}, we arrive at the following expression for $(\mathbf{A}_{\kappa})_{i,j}$:
\begin{equation} \label{eqn:element_matrix_reconstruction_formula}
    (\mathbf{A}_{\kappa})_{i,j} = \sum_{\bm{0}\le\bm{\alpha}\le\bm{\alpha}^{(i)}+\bm{\alpha}^{(j)}} c_{\bm{\alpha}}^{(i,j)} \int_{\hat{\kappa}} \hat{\x}^{\bm{\alpha}} \d\hat{\x},
\end{equation}

\noindent where the coefficients $\{c_{\bm{\alpha}}^{(i,j)}\}_{\bm{0}\le\bm{\alpha}\le\bm{\alpha}^{(i)}+\bm{\alpha}^{(j)}}$ are defined for each $1\le i,j\le \dim\mathbb{P}^p(\kappa)$ by
\begin{equation} \label{eqn:transport_monom_coeff_final}
    c_{\bm{\alpha}}^{(i,j)} = \left( c \prod_{k=1}^d C_{\alpha_k^{(i)},\alpha_k^{(j)},\alpha_k} - \sum_{k=1}^d \hat{b}_{\kappa,k} C_{\alpha_k^{(i)},\alpha_k^{(j)},\alpha_k}' \prod_{\substack{\ell=1 \\ \ell\ne k}}^d C_{\alpha_\ell^{(i)},\alpha_\ell^{(j)},\alpha_\ell} \right) |\mathbf{J}_\kappa|.
\end{equation}

\noindent We remark that the entries of $\mathbf{A}_\kappa$ are now in a form in which Algorithm \ref{alg:monomial_integration} can be applied to generate the set of integrated monomials $\ifun(\hat{\kappa},\mathcal{J})$. Here, we make the choice
\begin{equation*}
    \mathcal{J} = \left\{ \bm{\alpha}\in\naturals^d : 0\le|\bm{\alpha}|\le 2p \right\};
\end{equation*}

\noindent this ensures that $\mathcal{J}$ contains each $\bm{\alpha}$ for which $c_{\bm{\alpha}}^{(i,j)}\ne 0$. Moreover, by \reff{eqn:element_matrix_reconstruction_formula}, the set $\ifun(\hat{\kappa},\mathcal{J})$ can be computed once for each $\kappa$ and re-used to assemble each entry of $\mathbf{A}_\kappa$.

\begin{algorithm}[t!]
\caption{Computation of $\mathbf{A}_\kappa$ via quadrature-free integration.}
\label{alg:stiffness_assembly_quadfree}
\begin{algorithmic}[1]
    \Procedure{ComputeElementMatrix}{$\kappa$}
        \State Compute sets of coefficients $\{C_{m,n,k}\}$ and $\{C_{m,n,k}'\}$ in \reff{eqn:decomp_prod_leg_poly} and \reff{eqn:decomp_prod_leg_poly_deriv} (if not already available)
        \State Map $\kappa\mapsto\hat{\kappa}$
        \State Compute $\ifun(\hat{\kappa},\mathcal{J})$ using Algorithm \ref{alg:monomial_integration}
        \State Compute $\hat{\mathbf{b}}_\kappa=\mathbf{J}_\kappa^{-1}\mathbf{b}$
        \State $\mathbf{A}_\kappa\gets0$
        \For{$i=1,\dots,\dim\mathbb{P}^p(\kappa)$}
            \For{$j=1,\dots,\dim\mathbb{P}^p(\kappa)$}
                \For{$\bm{0}\le\bm{\alpha}\le\bm{\alpha}^{(i)}+\bm{\alpha}^{(j)}$}
                    \State Compute $c_{\bm{\alpha}}^{(i,j)}$ as in \reff{eqn:transport_monom_coeff_final}
                    \State $(\mathbf{A}_\kappa)_{i,j} \gets (\mathbf{A}_\kappa)_{i,j} + c_{\bm{\alpha}}^{(i,j)} \ifun(\hat{\kappa},\bm{\alpha})$
                \EndFor
            \EndFor
        \EndFor
        \State \Return $\mathbf{A}_\kappa$
    \EndProcedure
\end{algorithmic}
\end{algorithm}

Algorithm \ref{alg:stiffness_assembly_quadfree} provides pseudocode for a typical implementation of the quadrature-free based integration method to evaluate the integrals \reff{eqn:local_matrix_component_wise_defn} appearing in the DGFEM discretisation of the transport equation. We remark that the computational complexity associated with the execution of Algorithm~\ref{alg:monomial_integration} on line 4 has already been discussed in Section \ref{section:monomial_integration_quad_free_analysis_subsec}. Noting that lines 10 and 11 of Algorithm \ref{alg:stiffness_assembly_quadfree} require $(d+1)^2$ and 2 floating-point operations, respectively, it can be seen that the number of floating-point operations performed in the main body (lines 7-14) is given by $((d+1)^2+2) Q_d(p)$, where
\begin{equation} \label{eqn:Qd_fn_quad_free_assembly}
    Q_d(p) = \sum_{0\le|\bm{\alpha}^{(i)}|\le p} \sum_{0\le|\bm{\alpha}^{(j)}|\le p} \sum_{\bm{0}\le\bm{\alpha}\le\bm{\alpha}^{(i)}+\bm{\alpha}^{(j)}} 1.
\end{equation}

\noindent It can be shown that $Q_d(p) \sim \frac{1}{(3d)!} \binom{4d}{2d} p^{3d}$ as $p\rightarrow\infty$.

\begin{remark}[Pre-computation of coefficients]\label{remark:stiffness_qfree_coeff_precompute}
    One may optionally pre-compute the ($\kappa$-independent) coefficients $C_{\bm{\alpha}^{(i)},\bm{\alpha}^{(j)},\bm{\alpha}} = \prod_{k=1}^d C_{\alpha_k^{(i)},\alpha_k^{(j)},\alpha_k}$ and $C_{\bm{\alpha}^{(i)},\bm{\alpha}^{(j)},\bm{\alpha}}^{(k)} = C_{\alpha_k^{(i)},\alpha_k^{(j)},\alpha_k}' \prod_{\substack{\ell=1\\ \ell\ne k}}^d C_{\alpha_\ell^{(i)},\alpha_\ell^{(j)},\alpha_\ell}$ for $0\le|\bm{\alpha}^{(i)}|\le p$, $0\le|\bm{\alpha}^{(j)}|\le p$, $\bm{0}\le\bm{\alpha}\le\bm{\alpha}^{(i)}+\bm{\alpha}^{(j)}$ and $1\le k\le d$. This allows \reff{eqn:transport_monom_coeff_final} to be computed using $2(d+1)$ floating-point operations, which is the same as the number of floating-point operations needed to evaluate $I$ in the quadrature based implementation given in Algorithm \ref{alg:stiffness_assembly_quadrature}.
\end{remark}

\begin{remark}[$p$-refinement in quadrature-free based assembly]\label{remark:hierarchical_assembly}
    Algorithm \ref{alg:stiffness_assembly_quadfree} requires the following sets in order to assemble the matrix $\mathbf{A}_\kappa$ using a basis of $\mathbb{P}^p(\kappa)$:
    \begin{align*}
        \mathcal{C}(p) &= \left\{ C_{i,j,k} : 0\le i\le p, 0\le j\le p, 0\le k\le i+j \right\}, \\
        \mathcal{C}'(p) &= \left\{ C_{i,j,k}' : 0\le i\le p, 0\le j\le p, 0\le k\le i+j \right\}, \\
        \ifun_\kappa(p) &= \left\{ \ifun(\facet,\bm{\alpha}) : 0\le \dim\facet\le d, 0\le| \bm{\alpha} |\le 2p \right\}.
    \end{align*}

    Suppose we wish to perform a $p$-refinement; that is, to construct $\mathbf{A}_\kappa$ using a basis of $\mathbb{P}^{p+1}(\kappa)$. Further, assume that $\mathcal{C}(p)$, $\mathcal{C}'(p)$ and $\ifun_\kappa(p)$ are already known. We have that
    \begin{align*}
        \mathcal{C}(p+1) &= \mathcal{C}(p) \cup \left\{C_{i,p,k} : 0\le i\le p+1, \ 0\le k\le p+i \right\} \\
        &\quad\quad \cup \left\{C_{p,i,k} : 0\le i\le p+1, \ 0\le k\le p+i \right\}, \\
        \mathcal{C}'(p+1) &= \mathcal{C}'(p) \cup \left\{C_{i,p,k}' : 0\le i\le p+1, \ 0\le k\le p+i \right\} \\
        &\quad\quad \cup \left\{C_{p,i,k}' : 0\le i\le p+1, \ 0\le k\le p+i \right\}, \\
        \ifun_\kappa(p+1) &= \ifun_\kappa(p) \cup \left\{ \ifun(\facet,\bm{\alpha}) : 0\le \dim\facet\le d, \ 2p+1\le|\bm{\alpha}|\le2p+2 \right\}.
    \end{align*}

    \noindent The computations of the updated coefficient sets $\mathcal{C}(p+1)$ and $\mathcal{C}'(p+1)$ are a one-time cost (since these sets may be used for every $\kappa\in\mathcal{T}$) and the computation of the updated integral set $\ifun_\kappa(p+1)$ can be performed using a modification of Algorithm \ref{alg:monomial_integration} that uses prior knowledge of $\ifun_\kappa(p)$ to avoid unnecessary re-computation.
\end{remark}

\subsection{Comparison of assembly methods}

It can be seen that the time complexities associated with the main loops in Algorithms \ref{alg:stiffness_assembly_quadrature} and \ref{alg:stiffness_assembly_quadfree}, measured as the total number of floating-point operations required to assemble $\mathbf{A}_\kappa$, are $\bigO(p^{3d})$ in the limit as $p\rightarrow\infty$. Moreover, the time complexity associated with the execution of Algorithm \ref{alg:monomial_integration}, used to compute $\ifun(\hat{\kappa},\mathcal{J})$ in Algorithm \ref{alg:stiffness_assembly_quadfree}, is $\bigO(\chi_1(\kappa) p^d)$, where $\chi_1(\polytope)$ denotes the measure of complexity of the geometry of a polytope $\polytope$ given in the statement of Theorem \ref{thm:monomial_integral_alg_complexity}. Thus, for large enough $p$, the main loops in Algorithms \ref{alg:stiffness_assembly_quadrature} and \ref{alg:stiffness_assembly_quadfree} are the most expensive contributions to the total assembly time.

A study of the loops in Algorithms \ref{alg:stiffness_assembly_quadrature} and \ref{alg:stiffness_assembly_quadfree} shows that the quadrature based method reaches the inner-most computations $(\dim\mathbb{P}^p(\kappa))^2 N$ times, while the quadrature-free based method reaches the inner-most computations $Q_d(p)$ times, where $Q_d(p)$ is the function defined in \reff{eqn:Qd_fn_quad_free_assembly}. Here, the number of quadrature points and weights $N$ used in Algorithm \ref{alg:stiffness_assembly_quadrature} is chosen to exactly evaluate the integrals $(\mathbf{A}_\kappa)_{ij}$ in \reff{eqn:local_matrix_component_wise_defn} whose integrands are polynomial functions of total degree at most $2p$. Thus, a lower bound for $N$ is given by $N_{opt} = \binom{p+d}{d}$.

\begin{figure}[t!]
	\centering
	\begin{tikzpicture}
    \begin{groupplot}[group style={group size= 2 by 2},height=5cm,width=6.4cm]
        \nextgroupplot[xmode=log, log basis x={2},
					   ymode=log, log basis y={10},
					   ytick={1e0,1e2,1e4,1e6,1e8,1e10},
					   ymin=1e0, ymax=1e10,
                      xticklabel={\pgfmathparse{2^(\tick)}\pgfmathprintnumber{\pgfmathresult}},
					   ylabel=Inner-most computations, 
					   axis background/.style={fill=gray!0}, 
					   legend pos=north west,
					   grid=both,
					   grid style={line width=.1pt, draw=gray!10},
   					   major grid style={line width=.2pt,draw=gray!50}]
           
   				\addplot+[mark=square, color=tolcol1, samples=20, 
                domain = 1:20] {((x+1)*(x+2)/2)^3}; \label{plots:flop_count_quadrature}
       
   				\addplot+[mark=o, color=tolcol2, samples=20, 
                domain = 1:20] {((x+1)*(x+2)/2)^2 * (7*x^2+21*x+18)/18}; \label{plots:flop_count_quadfree}

                \addplot+[mark=none, color=black, dashed, samples=20, domain = 1:20] {x^6/10}; \label{plots:flop_count_O_p_3d}
       
                \coordinate (top) at (rel axis cs:0,1);
        \nextgroupplot[xmode=log, log basis x={2},
			     	   ymode=log, log basis y={10},
					   ytick={1e0,1e2,1e4,1e6,1e8,1e10},
					   ymin=1e0, ymax=1e10,
                      xticklabel={\pgfmathparse{2^(\tick)}\pgfmathprintnumber{\pgfmathresult}},
					   axis background/.style={fill=gray!0}, 
					   legend pos=north west,
					   grid=both,
					   grid style={line width=.1pt, draw=gray!10},
   					   major grid style={line width=.2pt,draw=gray!50}]
           
   				\addplot+[mark=square, color=tolcol1, samples=20, 
                domain = 1:20] {((x+1)*(x+2)*(x+3)/6)^3}; 
       
   				\addplot+[mark=o, color=tolcol2, samples=20, 
                domain = 1:20] {((x+1)*(x+2)*(x+3)/6)^2 * (x+2)*(11*x^2+44*x+60)/120}; 

                \addplot+[mark=none, color=black, dashed, samples=20, domain = 1:20] {x^9/500};
                
        \nextgroupplot[xmode=log, log basis x={2},
			     	   ymode=log, log basis y={10},
					   ytick={1e0,1e2,1e4,1e6,1e8,1e10},
					   ymin=1e0, ymax=1e11,
					   xlabel=$p$, 
                      xticklabel={\pgfmathparse{2^(\tick)}\pgfmathprintnumber{\pgfmathresult}},
					   ylabel=Floating-point operations, 
					   axis background/.style={fill=gray!0}, 
					   legend pos=north west,
					   grid=both,
					   grid style={line width=.1pt, draw=gray!10},
   					   major grid style={line width=.2pt,draw=gray!50}]
           
   				\addplot+[mark=square, color=tolcol1, samples=20, 
                domain = 1:20] {6*((x+1)*(x+2)/2)^3}; 
       
   				\addplot+[mark=o, color=tolcol2, samples=20, 
                domain = 1:20] {11*((x+1)*(x+2)/2)^2 * (7*x^2+21*x+18)/18}; 

                \addplot+[mark=none, color=black, dashed, samples=20, domain = 1:20] {x^6/5};
        \nextgroupplot[xmode=log, log basis x={2},
					   ymode=log, log basis y={10},
					   ytick={1e0,1e2,1e4,1e6,1e8,1e10},
					   ymin=1e0, ymax=1e11,
					   xlabel=$p$, 
                      xticklabel={\pgfmathparse{2^(\tick)}\pgfmathprintnumber{\pgfmathresult}},
					   axis background/.style={fill=gray!0}, 
					   legend pos=north west,
					   grid=both,
					   grid style={line width=.1pt, draw=gray!10},
   					   major grid style={line width=.2pt,draw=gray!50}]
           
   				\addplot+[mark=square, color=tolcol1, samples=20, 
                domain = 1:20] {8*((x+1)*(x+2)*(x+3)/6)^3}; 
       
   				\addplot+[mark=o, color=tolcol2, samples=20, 
                domain = 1:20] {18*((x+1)*(x+2)*(x+3)/6)^2 * (x+2)*(11*x^2+44*x+60)/120}; 

                \addplot+[mark=none, color=black, dashed, samples=20, domain = 1:20] {x^9/50};
       
                \coordinate (bot) at (rel axis cs:1,0);
    \end{groupplot}
    
	\path (top|-current bounding box.north)--
    	coordinate(legendpos)
    	(bot|-current bounding box.north);
	\matrix[
    	matrix of nodes,
    	anchor=south,
    	draw,
    	inner sep=0.2em,
    	draw
  	]at([yshift=1ex]legendpos)
  	{
    	\ref{plots:flop_count_quadrature}& Quadrature (Alg. \ref{alg:stiffness_assembly_quadrature}) &[5pt]
    	\ref{plots:flop_count_quadfree}& Quadrature-free (Alg. \ref{alg:stiffness_assembly_quadfree}) &[5pt]
    	\ref{plots:flop_count_O_p_3d}& $O(p^{3d})$ \\
    };
	\end{tikzpicture}
	\caption{Time complexities of the main loops in Algorithms \ref{alg:stiffness_assembly_quadrature} and \ref{alg:stiffness_assembly_quadfree} as a function of the degree of approximation $p$. It is assumed that $N=N_{opt}$ quadrature points are used in Algorithm \ref{alg:stiffness_assembly_quadrature}. Top row: number of times the operations within the main loops of Algorithms \ref{alg:stiffness_assembly_quadrature} and \ref{alg:stiffness_assembly_quadfree} are executed. Bottom row: total number of floating-point operations computed within the main loops of Algorithms \ref{alg:stiffness_assembly_quadrature} and \ref{alg:stiffness_assembly_quadfree}. Left column: $d=2$. Right column: $d=3$.}
	\label{fig:flop_count_quadrature_vs_quadfree}
\end{figure}
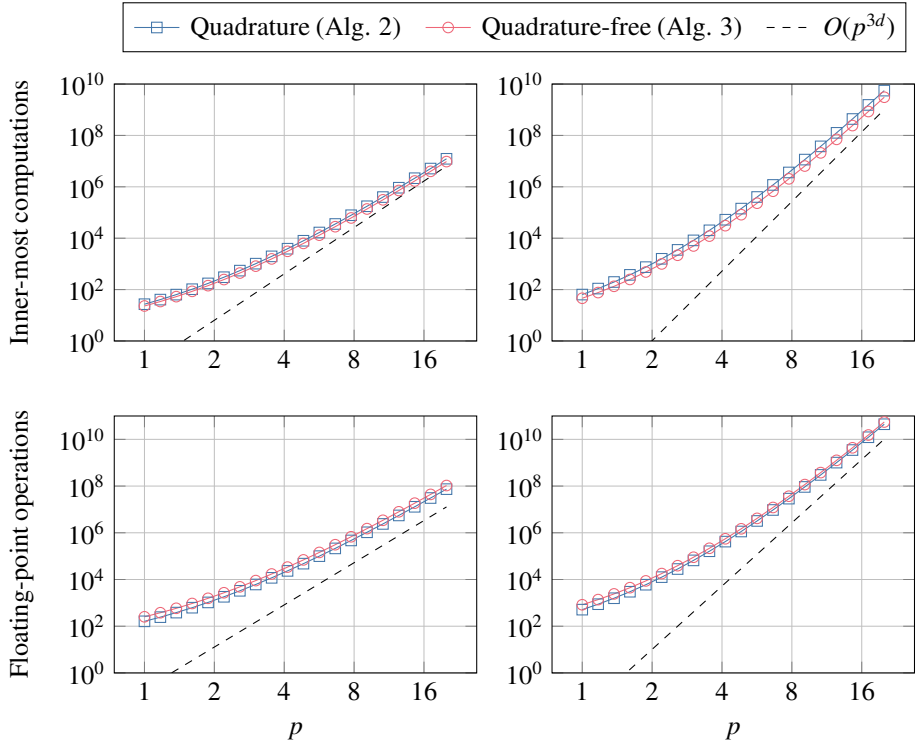

Figure \ref{fig:flop_count_quadrature_vs_quadfree} shows the expected number of floating-point operations required by Algorithm~\ref{alg:stiffness_assembly_quadrature} (using the theoretically-optimal number of quadrature points $N_{opt}$) and Algorithm~\ref{alg:stiffness_assembly_quadfree} used to assemble $\mathbf{A}_\kappa$. The leading-order behaviour of both algorithms in the limit $p\rightarrow\infty$ is $O(p^{3d})$. It is seen that, without taking the geometric complexity of $\kappa$ into consideration, the performance of the quadrature-free based assembly method is expected to be comparable to that of the quadrature based assembly employing the minimal quadrature set that exactly evaluates $\mathbf{A}_\kappa$.

However, the quadrature-free based assembly method can outperform the quadrature based assembly method when one takes into consideration the geometric complexity of $\kappa$. To see this, suppose that $\kappa$ is decomposed into $n$ subdomains for the purpose of numerical integration, on each of which an $N_{opt}$-point quadrature scheme can be applied. The leading-order behaviour of Algorithm \ref{alg:stiffness_assembly_quadrature} is $\bigO(np^{3d})$; that is, the time taken to assemble $\mathbf{A}_\kappa$ via the quadrature based assembly method will increase significantly if many integration subdomains are required. In contrast, the time taken to execute the main loop of Algorithm \ref{alg:stiffness_assembly_quadfree} is independent of $\kappa$; furthermore, taking into account the evaluation of the set of integrated monomials $\ifun(\hat{\kappa},\mathcal{J})$, the leading-order behaviour of Algorithm \ref{alg:stiffness_assembly_quadfree} is $\bigO(\chi_1(\hat{\kappa})p^d + p^{3d})$, where $\chi_1$ denotes the function given in the statement of Theorem \ref{thm:monomial_integral_alg_complexity}. We investigate this further in the following section.

\section{Numerical results} \label{sec:numerics}

In this section we consider the practical performance of the proposed quadrature-free algorithm.

\subsection{Effect of pruning in Algorithm \ref{alg:monomial_integration}}

\begin{table}[t]
    \centering
    \begin{tabular}{cc|ccccc}
        \multicolumn{2}{c|}{$|\mathcal{T}|$} & 32 & 128 & 512 & 2048 & 8192 \\ \hline
        \multirow{3}{*}{$p=0$} & Unpruned & 1.8926E-05 & 6.2850E-05 & 2.3568E-04 & 8.8175E-04 &3.6211E-03 \\
        & Pruned & 2.4888E-05 & 7.3720E-05 & 2.7315E-04 & 1.0641E-03 & 4.5114E-03 \\ 
        & Ratio & 1.32 & 1.17 & 1.16 & 1.21 & 1.25 \\ \hline
        \multirow{3}{*}{$p=2$} & Unpruned & 3.3402E-05 & 8.2150E-05 & 3.3563E-04 & 1.3032E-03 & 5.1413E-03 \\
        & Pruned & 2.8993E-05 & 8.0898E-05 & 3.1069E-04 & 1.2043E-03 & 4.7641E-03 \\ 
        & Ratio & 0.87 & 0.98 & 0.93 & 0.92 & 0.93 \\ \hline
        \multirow{3}{*}{$p=4$} & Unpruned & 3.3678E-05 & 1.2722E-04 & 5.1513E-04 & 1.7583E-03 & 6.4381E-03 \\
        & Pruned & 3.1136E-05 & 9.4751E-05 & 3.6298E-04 & 1.5314E-03 & 5.5922E-03 \\ 
        & Ratio & 0.92 & 0.74 & 0.70 & 0.87 & 0.87 \\ \hline
        \multirow{3}{*}{$p=6$} & Unpruned & 4.4647E-05 & 1.6264E-04 & 6.3257E-04 & 2.6062E-03 & 9.5336E-03 \\
        & Pruned & 3.8436E-05 & 1.2308E-04 & 5.2076E-04 & 1.9159E-03 & 7.2723E-03 \\ 
        & Ratio & 0.86 & 0.76 & 0.82 & 0.74 & 0.76 \\ \hline
        \multirow{3}{*}{$p=8$} & Unpruned & 6.4398E-05 & 2.2496E-04 & 9.3172E-04 & 3.2624E-03 & 1.2828E-02 \\
        & Pruned & 4.5497E-05 & 1.7977E-04 & 7.8489E-04 & 2.5205E-03 & 9.4220E-03 \\ 
        & Ratio & 0.71 & 0.80 & 0.84 & 0.77 & 0.73 \\ \hline
        \multirow{3}{*}{$p=10$} & Unpruned & 7.9538E-05 & 3.2234E-04 & 1.1563E-03 & 4.5004E-03 & 1.7702E-02 \\
        & Pruned & 5.6676E-05 & 2.2770E-04 & 7.9667E-04 & 3.0367E-03 & 1.2133E-02 \\ 
        & Ratio & 0.71 & 0.71 & 0.69 & 0.67 & 0.69 \\ \hline
        \multirow{3}{*}{$p=12$} & Unpruned & 1.3727E-04 & 4.1301E-04 & 1.6069E-03 & 5.7918E-03 & 2.3696E-02 \\
        & Pruned & 6.9956E-05 & 2.7148E-04 & 9.6113E-04 & 3.7934E-03 & 1.5079E-02 \\ 
        & Ratio & 0.51 & 0.66 & 0.60 & 0.65 & 0.64 \\ \hline
    \end{tabular}
    \caption{CPU times of unpruned and pruned versions of Algorithm \ref{alg:monomial_integration} for the assembly of $\bigcup_{\kappa\in\mathcal{T}}\ifun(\kappa,\mathcal{J})$ for triangular meshes $\mathcal{T}$ in two spatial dimensions.}
    \label{tab:transport_assembly_monom_int_pruning_tets_2d}
\end{table}

\begin{table}[t]
    \centering
    \begin{tabular}{cc|ccccc}
        \multicolumn{2}{c|}{$|\mathcal{T}|$} & 12 & 51 & 204 & 819 & 3276 \\ \hline
        \multirow{3}{*}{$p=0$} & Unpruned & 2.3893E-05 & 6.7615E-05 & 2.7051E-04 & 1.1335E-03 & 5.0569E-03 \\
        & Pruned & 3.1631E-05 & 8.4488E-05 & 3.1938E-04 & 1.3849E-03 & 6.2978E-03 \\ 
        & Ratio & 1.32 & 1.25 & 1.18 & 1.22 & 1.25 \\ \hline
        \multirow{3}{*}{$p=2$} & Unpruned & 2.9957E-05 & 1.0833E-04 & 3.4219E-04 & 1.6801E-03 & 6.6029E-03 \\
        & Pruned & 3.3662E-05 & 1.1878E-04 & 3.8369E-04 & 1.8276E-03 & 7.3026E-03 \\ 
        & Ratio & 1.12 & 1.10 & 1.12 & 1.09 & 1.11 \\ \hline
        \multirow{3}{*}{$p=4$} & Unpruned & 4.6595E-05 & 1.4459E-04 & 5.9718E-04 & 2.2960E-03 & 9.4577E-03 \\
        & Pruned & 3.9385E-05 & 1.2878E-04 & 5.9189E-04 & 2.2926E-03 & 9.2390E-03 \\ 
        & Ratio & 0.85 & 0.89 & 0.99 & 1.00 & 0.98 \\ \hline
        \multirow{3}{*}{$p=6$} & Unpruned & 5.6086E-05 & 2.1821E-04 & 9.3389E-04 & 3.2810E-03 & 1.2832E-02 \\
        & Pruned & 5.5423E-05 & 1.9397E-04 & 8.2830E-04 & 2.9239E-03 & 1.1699E-02 \\ 
        & Ratio & 0.99 & 0.89 & 0.89 & 0.89 & 0.91 \\ \hline
        \multirow{3}{*}{$p=8$} & Unpruned & 8.5570E-05 & 3.0688E-04 & 1.1368E-03 & 4.5906E-03 & 1.7794E-02 \\
        & Pruned & 8.1338E-05 & 2.5567E-04 & 9.8448E-04 & 3.8316E-03 & 1.5385E-02 \\ 
        & Ratio & 0.95 & 0.83 & 0.87 & 0.83 & 0.86 \\ \hline
        \multirow{3}{*}{$p=10$} & Unpruned & 1.2774E-04 & 4.2015E-04 & 1.5151E-03 & 6.0202E-03 & 2.4115E-02 \\
        & Pruned & 8.7737E-05 & 3.4614E-04 & 1.2627E-03 & 5.0212E-03 & 1.9915E-02 \\ 
        & Ratio & 0.69 & 0.82 & 0.83 & 0.83 & 0.83 \\ \hline
        \multirow{3}{*}{$p=12$} & Unpruned & 1.5864E-04 & 5.3115E-04 & 2.0558E-03 & 8.2953E-03 & 3.3196E-02 \\
        & Pruned & 1.4783E-04 & 4.3479E-04 & 1.7454E-03 & 6.7921E-03 & 2.5964E-02 \\ 
        & Ratio & 0.93 & 0.82 & 0.85 & 0.82 & 0.78 \\ \hline
    \end{tabular}
    \caption{CPU times of unpruned and pruned versions of Algorithm \ref{alg:monomial_integration} for the assembly of $\bigcup_{\kappa\in\mathcal{T}}\ifun(\kappa,\mathcal{J})$ for agglomerated triangular meshes $\mathcal{T}$ in two spatial dimensions.}
    \label{tab:transport_assembly_monom_int_pruning_agglomtets_2d}
\end{table}

\begin{table}[h!]
    \centering
    \begin{tabular}{cc|ccccc}
        %
        \multicolumn{2}{c|}{$|\mathcal{T}|$} & 6 & 48 & 384 & 3072 & 24576 \\ \hline
        \multirow{3}{*}{$p=0$} & Unpruned & 2.4926E-05 & 1.0816E-04 & 8.1476E-04 & 6.8287E-03 & 4.8521E-02 \\
        & Pruned & 1.9337E-05 & 7.1661E-05 & 5.1822E-04 & 4.0497E-03 & 3.1205E-02 \\ 
        & Ratio & 0.78 & 0.66 & 0.64 & 0.59 & 0.64 \\ \hline
        \multirow{3}{*}{$p=2$} & Unpruned & 3.5647E-05 & 2.0676E-04 & 1.5616E-03 & 1.3476E-02 & 9.5595E-02 \\
        & Pruned & 2.1014E-05 & 9.8906E-05 & 7.1403E-04 & 5.6071E-03 & 4.3727E-02 \\
        & Ratio & 0.59 & 0.48 & 0.46 & 0.42 & 0.46 \\ \hline
        \multirow{3}{*}{$p=4$} & Unpruned & 6.3102E-05 & 4.2202E-04 & 3.2991E-03 & 2.5008E-02 & 1.9898E-01 \\
        & Pruned & 3.1368E-05 & 1.5550E-04 & 1.2420E-03 & 9.1364E-03 & 7.2012E-02 \\
        & Ratio & 0.50 & 0.37 & 0.38 & 0.37 & 0.36 \\ \hline
        \multirow{3}{*}{$p=6$} & Unpruned & 1.1175E-04 & 8.2232E-04 & 5.9819E-03 & 4.7363E-02 & 3.7750E-01 \\
        & Pruned & 4.4362E-05 & 2.5484E-04 & 1.8026E-03 & 1.3952E-02 & 1.1175E-01 \\
        & Ratio & 0.40 & 0.31 & 0.30 & 0.29 & 0.30 \\ \hline
        \multirow{3}{*}{$p=8$} & Unpruned & 1.8889E-04 & 1.4290E-03 & 1.0750E-02 & 8.6211E-02 & 6.8734E-01 \\
        & Pruned & 6.2170E-05 & 2.8795E-04 & 3.0248E-03 & 2.3477E-02 & 1.8765E-01 \\
        & Ratio & 0.33 & 0.20 & 0.28 & 0.27 & 0.27 \\
        %
    \end{tabular}
    \caption{CPU times of unpruned and pruned versions of Algorithm \ref{alg:monomial_integration} for the assembly of $\bigcup_{\kappa\in\mathcal{T}}\ifun(\kappa,\mathcal{J})$ for tetrahedral meshes $\mathcal{T}$ in three spatial dimensions.}
    \label{tab:transport_assembly_monom_int_pruning_tets_3d}
\end{table}

\begin{table}[h!]
    \centering
    \begin{tabular}{cc|ccccc}
        %
        \multicolumn{2}{c|}{$|\mathcal{T}|$} & 4 & 38 & 307 & 2457 & 19660 \\ \hline
        \multirow{3}{*}{$p=0$} & Unpruned & 4.5272E-05 & 3.5574E-04 & 2.6056E-03 & 2.1919E-02 & 1.7550E-01 \\
        & Pruned & 3.7088E-05 & 2.8895E-04 & 2.0201E-03 & 1.6717E-02 & 1.3413E-01 \\ 
        & Ratio & 0.82 & 0.81 & 0.78 & 0.76 & 0.76 \\ \hline
        \multirow{3}{*}{$p=2$} & Unpruned & 8.7692E-05 & 7.7001E-04 & 5.5442E-03 & 4.3656E-02 & 3.5131E-01 \\
        & Pruned & 5.4265E-05 & 4.5130E-04 & 3.1994E-03 & 2.6491E-02 & 2.0647E-01 \\ 
        & Ratio & 0.62 & 0.59 & 0.58 & 0.61 & 0.59 \\ \hline
        \multirow{3}{*}{$p=4$} & Unpruned & 1.9519E-04 & 1.5893E-03 & 1.1656E-02 & 9.1985E-02 & 7.4515E-01 \\
        & Pruned & 9.1346E-05 & 7.6676E-04 & 5.5029E-03 & 4.4103E-02 & 3.5562E-01 \\ 
        & Ratio & 0.47 & 0.48 & 0.47 & 0.48 & 0.48 \\ \hline
        \multirow{3}{*}{$p=6$} & Unpruned & 3.4222E-04 & 3.0026E-03 & 2.3204E-02 & 1.8303E-01 & 1.4867E+00 \\
        & Pruned & 1.6222E-04 & 1.2524E-03 & 9.9395E-03 & 7.9390E-02 & 6.4008E-01 \\ 
        & Ratio & 0.47 & 0.42 & 0.43 & 0.43 & 0.43 \\ \hline
        \multirow{3}{*}{$p=8$} & Unpruned & 6.0944E-04 & 5.2648E-03 & 4.1087E-02 & 3.2613E-01 & 2.6380E+00 \\
        & Pruned & 2.4756E-04 & 1.9114E-03 & 1.5453E-02 & 1.2295E-01 & 1.0089E+00 \\ 
        & Ratio & 0.41 & 0.36 & 0.38 & 0.38 & 0.38 \\
        %
    \end{tabular}
    \caption{CPU times of unpruned and pruned versions of Algorithm \ref{alg:monomial_integration} for the assembly of $\bigcup_{\kappa\in\mathcal{T}}\ifun(\kappa,\mathcal{J})$ for agglomerated tetrahedral meshes $\mathcal{T}$ in three spatial dimensions.}
    \label{tab:transport_assembly_monom_int_pruning_agglomtets_3d}
\end{table}

We shall first study the effect of implementing Algorithm \ref{alg:monomial_integration} with and without pruning as described in Section \ref{section:monomial_integration_quad_free_analysis_subsec}. The pruning strategy we adopt is to select the reference point $\x_\facet$ as the first vertex of $\facet$ for each face visited by Algorithm \ref{alg:monomial_integration}. We will apply Algorithm \ref{alg:monomial_integration} to the problem of assembling the integral sets $\bigcup_{\kappa\in\mathcal{T}}\ifun(\kappa,\mathcal{J})$ in the case where $\mathcal{T}$ is a simplicial or agglomerated simplicial mesh in two or three spatial dimensions and
\begin{equation*}
    \mathcal{J} = \left\{ \bm{\alpha}\in\naturals^d : 0\le|\bm{\alpha}|\le p \right\}
\end{equation*}

\noindent for $p\in\{0,2,4,6,8,10,12\}$.

Tables~\ref{tab:transport_assembly_monom_int_pruning_tets_2d} and~\ref{tab:transport_assembly_monom_int_pruning_agglomtets_2d} show the total CPU time taken by the unpruned and pruned versions of Algorithm \ref{alg:monomial_integration} applied to two-dimensional triangular and agglomerated triangular meshes, respectively, while Tables~\ref{tab:transport_assembly_monom_int_pruning_tets_3d} and~\ref{tab:transport_assembly_monom_int_pruning_agglomtets_3d} show the total CPU time taken by the unpruned and pruned versions of Algorithm \ref{alg:monomial_integration} applied to three-dimensional tetrahedral and agglomerated tetrahedral meshes, respectively. An additional quantity, computed as the ratio of the CPU time taken by the pruned algorithm against the unpruned algorithm, is also reported; values of this ratio less than 1 indicate that Algorithm \ref{alg:monomial_integration} with pruning computes the integral set $\ifun(\polytope,\mathcal{J})$ faster than the same algorithm without pruning.

For each fixed $p$, it is observed that the ratio of CPU time taken by the pruned algorithm against the unpruned algorithm remains roughly constant in all cases. For a fixed number of elements, this ratio decreases as $p$ increases. It is expected that this ratio continues to decrease as $p\rightarrow\infty$ but remains bounded from below by a constant - for a single element $\kappa$, this constant is expected to depend on the number of nodes and edges of the graph $G(\kappa)$ defined in Section \ref{section:monomial_integration_quad_free_analysis_subsec} before and after pruning.

While pruning accelerates the assembly of $\bigcup_{\kappa\in\mathcal{T}}\ifun(\kappa,\mathcal{J})$ for both tetrahedral and agglomerated tetrahedral elements in three dimensions, a greater improvement in assembly time is observed for tetrahedral meshes - this is because a greater proportion of the nodes and edges of $G(\kappa)$ can be eliminated through pruning when $\kappa$ is simplicial. We observe similar behaviour between the pruned and unpruned version of Algorithm~\ref{alg:monomial_integration} applied to triangular and agglomerated triangular elements in two dimensions, though pruning is seen to be less effective at reducing the CPU time spent assembling $\bigcup_{\kappa\in\mathcal{T}} \ifun(\kappa,\mathcal{J})$. For small values of $p$, pruning actually slows down Algorithm \ref{alg:monomial_integration} - we speculate that this is because the extra time spent checking whether a given facet $\facet$ is to be pruned outweighs the expense that would be incurred to assemble $\ifun(\facet,\mathcal{J})$ without pruning.

\subsection{Integrating monomials over polygons}

As a second example, we compare the quadrature and \linebreak quadrature-free based integration algorithms for the evaluation of the sets
\begin{equation*}
    \ifun_{n,p} = \left\{ \int_{\polytope_n} \mathbf{x}^{\bm{\alpha}} \d\mathbf{x} : \bm{\alpha}\in\naturals^2, 0\le|\bm{\alpha}|\le p \right\},
\end{equation*}

\noindent where $\polytope_n\subset\reals^2$ denotes the regular $n$-gon, $5\le n\le 16$, with vertices $\left\{ \left( \cos\frac{2\pi k}{n},\sin\frac{2\pi k}{n} \right) \right\}_{k=0}^{n-1}$ and $p\in\{2,4,8,16,32\}$.

The quadrature based method is employed as follows: a sub-tessellation of $\polytope_n$ consisting of $(n-2)$ triangles is constructed by joining the first vertex of $\polytope_n$ to every other vertex. On each triangle, a $(q+1)^2$-point quadrature scheme, $q=\lceil \frac{p+1}{2} \rceil$, is defined by constructing a quadrature scheme on the unit square $(-1,1)^2$ exactly integrating all bivariate polynomial functions of maximal degree $p+1$. The reference quadrature scheme on $(-1,1)^2$ is then mapped to each triangle in the sub-tessellation of $\polytope_n$ via a Duffy transformation \cite{duffy1982quadrature}. The resulting quadrature scheme on $\polytope_n$ therefore contains $(n-2)(q+1)^2$ points and weights. We record the time taken for the quadrature based integration method to be executed for each element of $\ifun_{n,p}$; here, we do not include the time taken to generate the quadrature scheme on $\polytope_n$.

The quadrature-free based method is tailored to the two-dimensional setting. The integrals $\ifun(\polytope_n,\bm{\alpha})$ and $\ifun(\facet_{n,k},\bm{\alpha})$, for each boundary facet $\facet_{n,k}\subset\partial\polytope_n$, $1\le k\le n$, are stored in two arrays. We employ pruning based on selecting $\x_\facet$ as the first vertex of each visited face $\facet$.

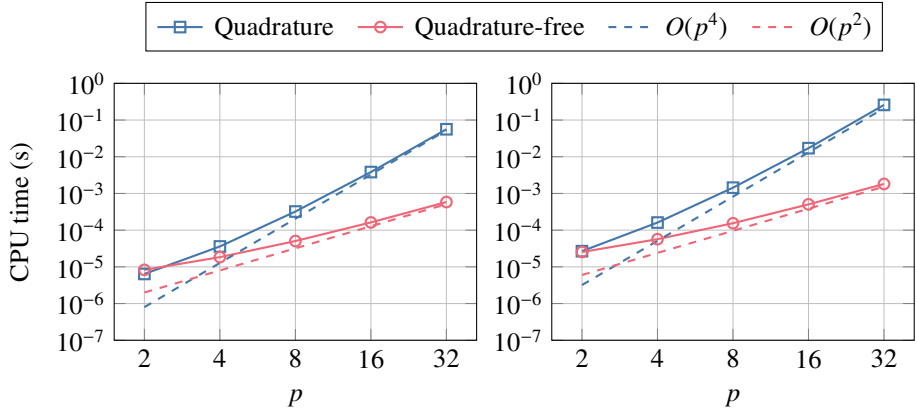
\begin{figure}[t!]
	\centering
	\begin{tikzpicture}
    \begin{groupplot}[group style={group size= 2 by 1},height=5cm,width=6.4cm]
        \nextgroupplot[xmode=log,
					   log basis x={2},
					   xticklabel={\pgfmathparse{2^(\tick)}\pgfmathprintnumber{\pgfmathresult}},
					   ymode=log,
					   log basis y={10},
					   ytick={1e-7,1e-6,1e-5,1e-4,1e-3,1e-2,1e-1,1e0},
					   ymin=1e-7, ymax=1e0,
					   xlabel=$p$, 
					   ylabel=CPU time (s), 
					   axis background/.style={fill=gray!0}, 
					   legend pos=north west,
					   grid=both,
					   grid style={line width=.1pt, draw=gray!10},
   					   major grid style={line width=.2pt,draw=gray!50}]
   				\addplot+[mark=square, thick, tolcol1, mark options={tolcol1, solid}] table [x=p, y=t_quad, col sep=comma] {data/monomial_times_2d/monomial_timings_n05.csv}; \label{plots:monomial_times_x_poly_n_5_quad_2d}
   				\addplot+[mark=o, thick, tolcol2, mark options={tolcol2, solid}] table [x=p, y=t_qfree, col sep=comma] {data/monomial_times_2d/monomial_timings_n05.csv}; \label{plots:monomial_times_x_poly_n_5_qfree_2d}
   				\addplot+[mark=none, tolcol1, dashed, thick, domain=2:32] {5e-8*x^4}; \label{plots:monomial_times_x_poly_Op4_ref_2d}
   				\addplot+[mark=none, tolcol2, dashed, thick, domain=2:32] {5e-7*x^2}; \label{plots:monomial_times_x_poly_Op2_ref_2d}
                \coordinate (top) at (rel axis cs:0,1);
        \nextgroupplot[xmode=log,
					   log basis x={2},
					   xticklabel={\pgfmathparse{2^(\tick)}\pgfmathprintnumber{\pgfmathresult}},
					   ymode=log,
					   log basis y={10},
					   ytick={1e-7,1e-6,1e-5,1e-4,1e-3,1e-2,1e-1,1e0},
					   ymin=1e-7, ymax=1e0,
					   xlabel=$p$, 
					   axis background/.style={fill=gray!0}, 
					   legend pos=north west,
					   grid=both,
					   grid style={line width=.1pt, draw=gray!10},
   					   major grid style={line width=.2pt,draw=gray!50}]
   					   minor grid style={line width=.1pt,draw=gray!10}]
   				\addplot+[mark=square, thick, tolcol1, mark options={tolcol1, solid}] table [x=p, y=t_quad, col sep=comma] {data/monomial_times_2d/monomial_timings_n16.csv};
   				\addplot+[mark=o, thick, tolcol2, mark options={tolcol2, solid}] table [x=p, y=t_qfree, col sep=comma] {data/monomial_times_2d/monomial_timings_n16.csv};
   				\addplot+[mark=none, tolcol1, dashed, thick, domain=2:32] {2e-7*x^4};
   				\addplot+[mark=none, tolcol2, dashed, thick, domain=2:32] {1.5e-6*x^2};
                \coordinate (bot) at (rel axis cs:1,0);
    \end{groupplot}
	\path (top|-current bounding box.north)--
    	coordinate(legendpos)
    	(bot|-current bounding box.north);
	\matrix[
    	matrix of nodes,
    	anchor=south,
    	draw,
    	inner sep=0.2em,
    	draw
  	]at([yshift=1ex]legendpos)
  	{
    	\ref{plots:monomial_times_x_poly_n_5_quad_2d}& Quadrature &[5pt]
    	\ref{plots:monomial_times_x_poly_n_5_qfree_2d}& Quadrature-free &[5pt]
    	\ref{plots:monomial_times_x_poly_Op4_ref_2d}& $O(p^4)$ &[5pt]
    	\ref{plots:monomial_times_x_poly_Op2_ref_2d}& $O(p^2)$ \\};
	\end{tikzpicture}
	\caption{CPU times taken by the quadrature-based and quadrature-free-based methods to evaluate $\ifun_{n,p}$ for $p=2,4,8,16,32$ on a regular $n$-gon. Left: $n=5$. Right: $n=16$.}
	\label{fig:monomial_times_x_poly_2d}
\end{figure}

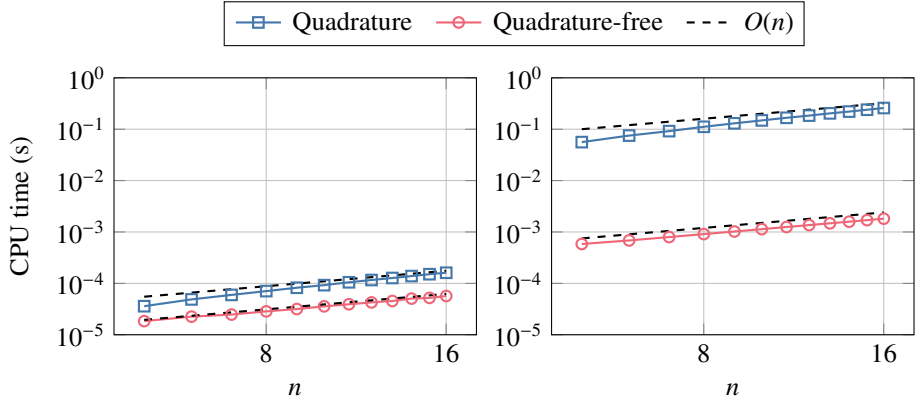
\begin{figure}[t!]
	\centering
	\begin{tikzpicture}
    \begin{groupplot}[group style={group size= 2 by 1},height=5cm,width=6.4cm]
        \nextgroupplot[xmode=log,
					   log basis x={2},
					   xticklabel={\pgfmathparse{2^(\tick)}\pgfmathprintnumber{\pgfmathresult}},
					   ymode=log,
					   log basis y={10},
					   ytick={1e-5,1e-4,1e-3,1e-2,1e-1,1e0},
					   ymin=1e-5, ymax=1e0,
					   xlabel=$n$, 
					   ylabel=CPU time (s), 
					   axis background/.style={fill=gray!0}, 
					   legend pos=north west,
					   grid=both,
					   grid style={line width=.1pt, draw=gray!10},
   					   major grid style={line width=.2pt,draw=gray!50}]
   				\addplot+[mark=square, thick, tolcol1, mark options={tolcol1, solid}] table [x=n, y=t_quad, col sep=comma] {data/monomial_times_2d/monomial_timings_p04.csv}; \label{plots:monomial_times_x_poly_p_4_quad_2d}
   				\addplot+[mark=o, thick, tolcol2, mark options={tolcol2, solid}] table [x=n, y=t_qfree, col sep=comma] {data/monomial_times_2d/monomial_timings_p04.csv}; \label{plots:monomial_times_x_poly_p_4_qfree_2d}
   				\addplot+[mark=none, black, dashed, thick, domain=5:16] {1.1e-5*x}; \label{plots:monomial_times_x_poly_On_ref_2d}
   				\addplot+[mark=none, black, dashed, thick, domain=5:16] {3.9e-6*x};
                \coordinate (top) at (rel axis cs:0,1);
        \nextgroupplot[xmode=log,
					   log basis x={2},
					   xticklabel={\pgfmathparse{2^(\tick)}\pgfmathprintnumber{\pgfmathresult}},
					   ymode=log,
					   log basis y={10},
					   ytick={1e-5,1e-4,1e-3,1e-2,1e-1,1e0},
					   ymin=1e-5, ymax=1e0,
					   xlabel=$n$, 
					   axis background/.style={fill=gray!0}, 
					   legend pos=north west,
					   grid=both,
					   grid style={line width=.1pt, draw=gray!10},
   					   major grid style={line width=.2pt,draw=gray!50}]
   					   minor grid style={line width=.1pt,draw=gray!10}]
   				\addplot+[mark=square, thick, tolcol1, mark options={tolcol1, solid}] table [x=n, y=t_quad, col sep=comma] {data/monomial_times_2d/monomial_timings_p32.csv};
   				\addplot+[mark=o, thick, tolcol2, mark options={tolcol2, solid}] table [x=n, y=t_qfree, col sep=comma] {data/monomial_times_2d/monomial_timings_p32.csv};
   				\addplot+[mark=none, black, dashed, thick, domain=5:16] {2e-2*x};
   				\addplot+[mark=none, black, dashed, thick, domain=5:16] {1.5e-4*x};
                \coordinate (bot) at (rel axis cs:1,0);
    \end{groupplot}
	\path (top|-current bounding box.north)--
    	coordinate(legendpos)
    	(bot|-current bounding box.north);
	\matrix[
    	matrix of nodes,
    	anchor=south,
    	draw,
    	inner sep=0.2em,
    	draw
  	]at([yshift=1ex]legendpos)
  	{
    	\ref{plots:monomial_times_x_poly_p_4_quad_2d}& Quadrature &[5pt]
    	\ref{plots:monomial_times_x_poly_p_4_qfree_2d}& Quadrature-free &[5pt]
    	\ref{plots:monomial_times_x_poly_On_ref_2d}& $O(n)$ \\};
	\end{tikzpicture}
	\caption{CPU times taken by the quadrature-based and quadrature-free-based methods to evaluate $\ifun_{n,p}$ for $5\le n\le 16$ and fixed $p$. Left: $p=4$. Right: $p=32$.}
	\label{fig:monomial_times_x_ngon_2d}
\end{figure}

Figures \ref{fig:monomial_times_x_poly_2d} and \ref{fig:monomial_times_x_ngon_2d} show the CPU times taken by the quadrature based and quadrature-free based integration algorithms to evaluate $\ifun_{n,p}$ averaged over 100 function calls. The time complexity of the quadrature based method is $\bigO(np^4)$, since the size of the requested set of integrals $|\ifun_{n,p}|=\bigO(p^2)$ and each integral requires $\bigO(n(q+1)^2)=\bigO(np^2)$ flops to compute. On the other hand, the quadrature-free based method is seen to have time complexity $\bigO(np^2)$. This is consistent with Theorem \ref{thm:monomial_integral_alg_complexity_d_2_3} since $|J|=\bigO(p^2)$. It has been verified that the integrals in the set $\ifun_{n,p}$ computed using Algorithm \ref{alg:monomial_integration} agree with the same set of integrals computed using quadrature to within machine precision.

\subsection{2D transport matrix assembly}

As a third example, we compare the quadrature based and quadrature-free based assembly methods for a single element matrix arising from the DGFEM discretisation of linear first-order transport problems posed in two spatial dimensions. We first consider a single $n$-sided polygonal domain $\polytope_n\subset\reals^2$ with $3\le n\le 64$ and employ a basis of degree $p=4$, i.e., the finite element space is $\mathbb{V} = \mathbb{P}^4(\polytope_n)$.

As before, the quadrature based method employs a sub-tessellation for the purposes of constructing a quadrature scheme on $\polytope_n$. In this example, the sub-tessellation consists of $n$ triangles constructed by joining each vertex of $\polytope_n$ to the centroid. Here, a $(p+2)^2$-point quadrature scheme is employed on each triangle using the method outlined in the previous example; this ensures that the quadrature scheme exactly evaluates the element integrals appearing in \reff{eqn:dgfem_linear_transport}. As before, we do not include the time taken to generate the quadrature scheme on $\polytope_n$. The time taken to evaluate the basis functions at the quadrature points is also excluded.

\begin{figure}[t!]
	\centering
	\begin{tikzpicture}
    \begin{groupplot}[group style={group size= 2 by 1},height=5cm,width=6.4cm]
        \nextgroupplot[xmode=log,
					   log basis x={2},
					   xticklabel={\pgfmathparse{2^(\tick)}\pgfmathprintnumber{\pgfmathresult}},
					   ymode=log,
					   log basis y={10},
					   ytick={1e-7,1e-6,1e-5,1e-4,1e-3,1e-2},
					   ymin=1e-7, ymax=1e-2,
					   xlabel=$n$, 
					   ylabel=CPU time (s), 
					   axis background/.style={fill=gray!0}, 
					   legend pos=north west,
					   grid=both,
					   grid style={line width=.1pt, draw=gray!10},
   					   major grid style={line width=.2pt,draw=gray!50}]
           
                \addplot+[mark=square, thick, tolcol1, mark options={tolcol1, solid}] table [x=n_sides, y=volume_assembly_timer, col sep=comma] {data/transport_assembly_times_2d/transport_assembly_2d_quadrature_64_p_4.csv}; \label{plots:transport_times_2d_quadrature}
       
                \addplot+[mark=o, thick, tolcol2, mark options={tolcol2, solid}] table [x=n_sides, y=volume_assembly_timer, col sep=comma] {data/transport_assembly_times_2d/transport_assembly_2d_qfree_pruned_64_p_4.csv}; \label{plots:transport_times_2d_qfree_full}
       
   				\addplot+[mark=none, black, dashed, thick, domain=3:64] {1.e-4*x/4}; \label{plots:transport_times_2d_On}
       
                \coordinate (top) at (rel axis cs:0,1);
       
        \nextgroupplot[xmode=log,
					   log basis x={2},
					   xticklabel={\pgfmathparse{2^(\tick)}\pgfmathprintnumber{\pgfmathresult}},
					   ymode=log,
					   log basis y={10},
					   ytick={1e-7,1e-6,1e-5,1e-4,1e-3,1e-2},
					   ymin=1e-7, ymax=1e-2,
					   xlabel=$n$, 
					   axis background/.style={fill=gray!0}, 
					   legend pos=north west,
					   grid=both,
					   grid style={line width=.1pt, draw=gray!10},
   					   major grid style={line width=.2pt,draw=gray!50}]
   					   minor grid style={line width=.1pt,draw=gray!10}]
           
                \addplot+[mark=square, thick, tolcol1, mark options={tolcol1, solid}] table [x=n_sides, y=volume_assembly_timer, col sep=comma] {data/transport_assembly_times_2d/transport_assembly_2d_quadrature_64_p_4.csv}; 
       
                \addplot+[mark=diamond, thick, tolcol4, mark options={tolcol4, solid}] table [x=n_sides, y=integral_timer, col sep=comma] {data/transport_assembly_times_2d/transport_assembly_2d_qfree_pruned_64_p_3.csv}; \label{plots:transport_times_2d_qfree_monomials}
                \addplot+[mark=triangle, thick, tolcol3, mark options={tolcol3, solid}] table [x=n_sides, y expr=\thisrowno{5}-\thisrowno{4}, col sep=comma] {data/transport_assembly_times_2d/transport_assembly_2d_qfree_pruned_64_p_4.csv}; \label{plots:transport_times_2d_qfree_reconstruction}
       
   				\addplot+[mark=none, black, dashed, thick, domain=3:64] {1.e-4*x/4}; 
   				\addplot+[mark=none, black, dashed, thick, domain=3:64] {1.e-6*x/8}; 
       
                \coordinate (bot) at (rel axis cs:1,0);
    \end{groupplot}
	\path (top|-current bounding box.north)--
    	coordinate(legendpos)
    	(bot|-current bounding box.north);
	\matrix[
    	matrix of nodes,
    	anchor=south,
    	draw,
    	inner sep=0.2em,
    	draw
  	]at([yshift=1ex]legendpos)
  	{
    	\ref{plots:transport_times_2d_quadrature}& Quadrature (Alg. \ref{alg:stiffness_assembly_quadrature}) &[5pt]
    	\ref{plots:transport_times_2d_qfree_full}& Quad-free (Alg. \ref{alg:stiffness_assembly_quadfree}) &[5pt] \\
    	\ref{plots:transport_times_2d_qfree_monomials}& Integration (Alg. \ref{alg:stiffness_assembly_quadfree}, line 4) &[5pt]
    	\ref{plots:transport_times_2d_qfree_reconstruction}& Reconstruction (Alg. \ref{alg:stiffness_assembly_quadfree}, lines 7-14) &[5pt] \\
        \ref{plots:transport_times_2d_On}& $O(n)$ \\};
	\end{tikzpicture}
	\caption{CPU times taken by the quadrature-based and quadrature-free-based methods to evaluate the DGFEM element transport matrix for $3\le n\le 64$ and fixed $p=4$. Left: total time taken by quadrature-based and quadrature-free-based methods. Right: Contributions to CPU time arising from Algorithm \ref{alg:stiffness_assembly_quadrature} (blue), line 4 of Algorithm \ref{alg:stiffness_assembly_quadfree} (purple) and lines 7-14 of Algorithm \ref{alg:stiffness_assembly_quadfree} (green).}
	\label{fig:transport_times_ngon_2d}
\end{figure}
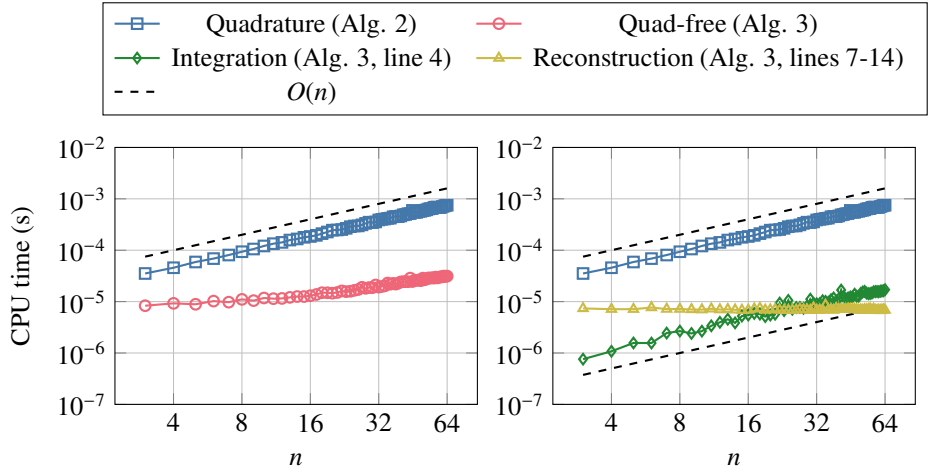

\begin{figure}[t!]
	\centering
	\begin{tikzpicture}
    \begin{groupplot}[group style={group size= 2 by 1},height=5cm,width=6.4cm]
        \nextgroupplot[xmode=log,
					   log basis x={2},
					   xticklabel={\pgfmathparse{2^(\tick)}\pgfmathprintnumber{\pgfmathresult}},
					   ymode=log,
					   log basis y={10},
					   ytick={1e-7,1e-6,1e-5,1e-4,1e-3,1e-2},
					   ymin=1e-7, ymax=1e-2,
					   xlabel=$p$, 
					   ylabel=CPU time (s), 
					   axis background/.style={fill=gray!0}, 
					   legend pos=north west,
					   grid=both,
					   grid style={line width=.1pt, draw=gray!10},
   					   major grid style={line width=.2pt,draw=gray!50}]
           
                \addplot+[mark=square, thick, tolcol1, mark options={tolcol1, solid}] table [x=pmax, y=volume_assembly_timer, col sep=comma] {data/transport_assembly_times_2d/transport_assembly_2d_quadrature_polytest.csv}; \label{plots:transport_times_2d_quadrature_polytest}
       
                \addplot+[mark=o, thick, tolcol2, mark options={tolcol2, solid}] table [x=pmax, y=volume_assembly_timer, col sep=comma] {data/transport_assembly_times_2d/transport_assembly_2d_qfree_pruned_polytest.csv}; \label{plots:transport_times_2d_qfree_full_polytest}
       
   				\addplot+[mark=none, black, dashed, thick, domain=1:12] {x^6*2e-3/2985984}; \label{plots:transport_times_2d_polytest_Op3d}
       
                \coordinate (top) at (rel axis cs:0,1);
       
        \nextgroupplot[xmode=log,
					   log basis x={2},
					   xticklabel={\pgfmathparse{2^(\tick)}\pgfmathprintnumber{\pgfmathresult}},
					   ymode=log,
					   log basis y={10},
					   ytick={1e-7,1e-6,1e-5,1e-4,1e-3,1e-2},
					   ymin=1e-7, ymax=1e-2,
					   xlabel=$p$, 
					   axis background/.style={fill=gray!0}, 
					   legend pos=north west,
					   grid=both,
					   grid style={line width=.1pt, draw=gray!10},
   					   major grid style={line width=.2pt,draw=gray!50}]
   					   minor grid style={line width=.1pt,draw=gray!10}]
           
                \addplot+[mark=square, thick, tolcol1, mark options={tolcol1, solid}] table [x=pmax, y=volume_assembly_timer, col sep=comma] {data/transport_assembly_times_2d/transport_assembly_2d_quadrature_polytest.csv}; 
       
                \addplot+[mark=o, thick, tolcol4, mark options={tolcol4, solid}] table [x=pmax, y=integral_timer, col sep=comma] {data/transport_assembly_times_2d/transport_assembly_2d_qfree_pruned_polytest.csv}; \label{plots:transport_times_2d_qfree_monomials_polytest}
                \addplot+[mark=o, thick, tolcol5, mark options={tolcol5, solid}] table [x=pmax, y expr=\thisrowno{5}-\thisrowno{4}, col sep=comma] {data/transport_assembly_times_2d/transport_assembly_2d_qfree_pruned_polytest.csv}; \label{plots:transport_times_2d_qfree_reconstruction_polytest}
       
   				\addplot+[mark=none, black, dashed, thick, domain=1:12] {x^6*2e-3/2985984}; 
   				\addplot+[mark=none, black, dash dot, thick, domain=1:12] {x^2*1e-5/144}; \label{plots:transport_times_2d_polytest_Opd}
       
                \coordinate (bot) at (rel axis cs:1,0);
    \end{groupplot}
	\path (top|-current bounding box.north)--
    	coordinate(legendpos)
    	(bot|-current bounding box.north);
	\matrix[
    	matrix of nodes,
    	anchor=south,
    	draw,
    	inner sep=0.2em,
    	draw
  	]at([yshift=1ex]legendpos)
  	{
    	\ref{plots:transport_times_2d_quadrature_polytest}& Quadrature (Alg. \ref{alg:stiffness_assembly_quadrature}) &[5pt]
    	\ref{plots:transport_times_2d_qfree_full_polytest}& Quad-free (Alg. \ref{alg:stiffness_assembly_quadfree}) &[5pt] \\
    	\ref{plots:transport_times_2d_qfree_monomials_polytest}& Integration (Alg. \ref{alg:stiffness_assembly_quadfree}, line 4) &[5pt]
    	\ref{plots:transport_times_2d_qfree_reconstruction_polytest}& Reconstruction (Alg. \ref{alg:stiffness_assembly_quadfree}, lines 7-14) &[5pt] \\
        \ref{plots:transport_times_2d_polytest_Op3d}& $O(p^{6})$ &[5pt] \ref{plots:transport_times_2d_polytest_Opd}& $O(p^2)$ \\};
	\end{tikzpicture}
	\caption{CPU times taken by the quadrature-based and quadrature-free-based methods to evaluate the DGFEM element transport matrix for fixed $n=6$ and $1\le p\le 12$. Left: total time taken by quadrature-based and quadrature-free-based methods. Right: Contributions to CPU time arising from Algorithm \ref{alg:stiffness_assembly_quadrature} (blue), line 4 of Algorithm \ref{alg:stiffness_assembly_quadfree} (purple) and lines 7-14 of Algorithm \ref{alg:stiffness_assembly_quadfree} (green).}
	\label{fig:transport_times_6gon_p_dep}
\end{figure}

The quadrature-free based method is tailored to the two-dimensional setting. The element integrals appearing in \reff{eqn:dgfem_linear_transport} are evaluated using the two-step procedure given in Algorithm \ref{alg:stiffness_assembly_quadfree}. The monomial integrals $\int_{\polytope_n} \x^{\bm{\alpha}} \d\x$ for $0\le|\bm{\alpha}|\le2p$ are computed once based on employing Algorithm \ref{alg:monomial_integration} using the first-vertex based pruning strategy as before; these are then used to assemble the local element matrix entry-wise through the decomposition \reff{eqn:element_matrix_reconstruction_formula} of the integrals given in \reff{eqn:dgfem_linear_transport}. The time taken to generate the coefficients in these decompositions is not included.

Figure \ref{fig:transport_times_ngon_2d} shows the CPU time taken by the quadrature based assembly method using Algorithm \ref{alg:stiffness_assembly_quadrature} and the quadrature-free based assembly method using Algorithms \ref{alg:stiffness_assembly_quadfree} to assemble the element matrices arising from a DGFEM discretisation of a linear, constant-coefficient transport problem in two spatial dimensions. The CPU times are averaged over 10000 calls to both assembly methods on the same $n$-gon element $\polytope_n$ for $3\le n\le 64$ on which a basis of $\mathbb{P}^4(\polytope_n)$ is employed. The time taken by Algorithm \ref{alg:stiffness_assembly_quadfree} to assemble the element matrices is further broken down into contributions from line 4 (i.e., Algorithm~\ref{alg:monomial_integration}) and lines 7-14 (i.e., the reconstruction of $(\mathbf{A}_\kappa)_{i,j}$ via \reff{eqn:element_matrix_reconstruction_formula}).

As seen in the previous example, the time-complexity of Algorithm \ref{alg:stiffness_assembly_quadrature} scales linearly with $n$, the number of sides of the element $\polytope_n$. The time-complexity associated with line 4 of Algorithm \ref{alg:stiffness_assembly_quadfree} - which we remark is a call to Algorithm \ref{alg:monomial_integration} also scales linearly with $n$, as expected.

The main body of Algorithm \ref{alg:stiffness_assembly_quadfree}, namely the nested loops on lines 7-14, is seen to scale independently of $n$. Indeed, it was seen in the analysis of the previous section that this loop exhibits no dependence on the geometry of the element. Therefore, the actual run-time of the quadrature-free based assembly method depends on whether the integration phase (line 4) or reconstruction phase (lines 7-14) is more expensive. For small values of $n$, the reconstruction phase is more expensive and so the assembly time remains roughly constant; for large enough $n$ the integration phase dominates the computational time and so a dependence of the assembly time on the geometric complexity of the element emerges.

We now consider fixing $n$ to study the dependence on $p$; to this end, in Figure \ref{fig:transport_times_6gon_p_dep} we show the scalings of Algorithms~\ref{alg:stiffness_assembly_quadrature} and~\ref{alg:stiffness_assembly_quadfree} as functions of the polynomial degree of approximation for $1\le p\le12$ on a fixed $6$-sided polygonal domain. For large enough $p$, it can be seen that the time complexities of the quadrature based and quadrature-free based assembly methods are both $\bigO(p^{3d})=\bigO(p^6)$ as predicted earlier. In the case of the quadrature-free based assembly method, the time complexities (on a given geometry) of the monomial integration phase (line 4) and the reconstruction phase (lines 7-14) of Algorithm \ref{alg:stiffness_assembly_quadfree} are $\bigO(p^{3d})=\bigO(p^6)$ and $\bigO(p^d)=\bigO(p^2)$, respectively. This is in agreement with our previous analysis, which predicted that the time complexity of the quadrature-free based assembly scales like $\bigO(\chi_1(\polytope_n)p^d + p^{3d}) = \bigO(\chi_1(\polytope_n)p^2+p^6)$, where $\chi_1(\polytope)$ denotes the measure of complexity of the geometry of a polytope $\polytope$ given in the statement of Theorem \ref{thm:monomial_integral_alg_complexity}. In this example, it is observed that the quadrature-free based assembly method is faster than the quadrature based assembly method by a factor of around an order of magnitude for all tested polynomial degrees.

\subsection{3D transport matrix assembly}

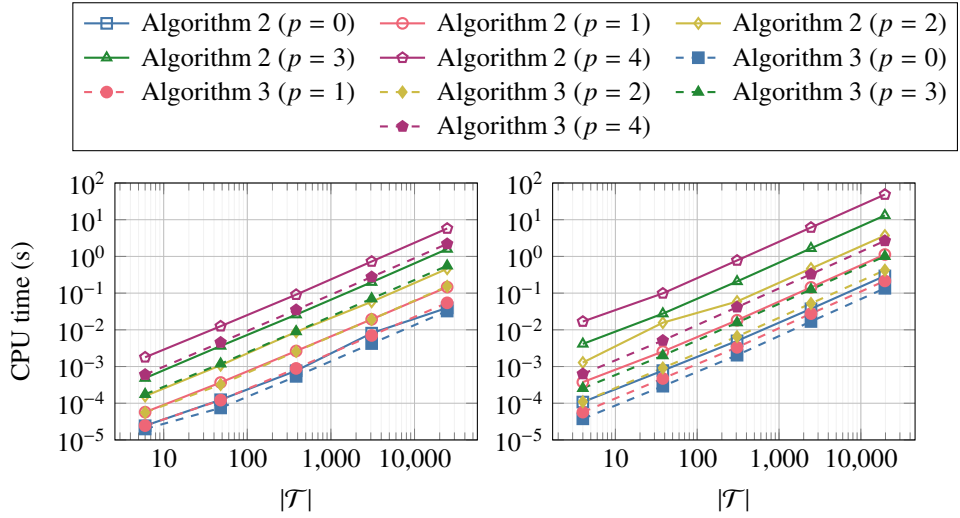
\begin{figure}[t!]
	\centering
	\begin{tikzpicture}
    \begin{groupplot}[group style={group size= 2 by 1},height=5cm,width=6.4cm]
        \nextgroupplot[xmode=log,
					   log basis x={10},
					   xticklabel={\pgfmathparse{10^(\tick)}\pgfmathprintnumber{\pgfmathresult}},
					   ymode=log,
					   log basis y={10},
					   ytick={1e-5,1e-4,1e-3,1e-2,1e-1,1e0,1e1,1e2},
					   ymin=1e-5, ymax=1e2,
					   xlabel=$|\mathcal{T}|$, 
					   ylabel=CPU time (s), 
					   axis background/.style={fill=gray!0}, 
					   legend pos=north west,
					   grid=both,
					   grid style={line width=.1pt, draw=gray!10},
   					   major grid style={line width=.2pt,draw=gray!50}]

            \addplot+[mark=square, thick, solid, tolcol1, mark options={tolcol1, solid}] table [x=no_eles, y=volume_assembly_timer, col sep=comma, restrict expr to domain={\thisrow{pmax}}{0:0}, unbounded coords=discard] {data/transport_assembly_times_3d/transport_assembly_3d_quadrature_tets.csv}; \label{plots:transport_times_3d_quadrature_p_0}
            
            \addplot+[mark=o, thick, solid, tolcol2, mark options={tolcol2, solid}] table [x=no_eles, y=volume_assembly_timer, col sep=comma, restrict expr to domain={\thisrow{pmax}}{1:1}, unbounded coords=discard] {data/transport_assembly_times_3d/transport_assembly_3d_quadrature_tets.csv}; \label{plots:transport_times_3d_quadrature_p_1}
            
            \addplot+[mark=diamond, thick, solid, tolcol3, mark options={tolcol3, solid}] table [x=no_eles, y=volume_assembly_timer, col sep=comma, restrict expr to domain={\thisrow{pmax}}{2:2}, unbounded coords=discard] {data/transport_assembly_times_3d/transport_assembly_3d_quadrature_tets.csv}; \label{plots:transport_times_3d_quadrature_p_2}
            
            \addplot+[mark=triangle, thick, solid, tolcol4, mark options={tolcol4, solid}] table [x=no_eles, y=volume_assembly_timer, col sep=comma, restrict expr to domain={\thisrow{pmax}}{3:3}, unbounded coords=discard] {data/transport_assembly_times_3d/transport_assembly_3d_quadrature_tets.csv}; \label{plots:transport_times_3d_quadrature_p_3}
            
            \addplot+[mark=pentagon, thick, solid, tolcol5, mark options={tolcol5, solid}] table [x=no_eles, y=volume_assembly_timer, col sep=comma, restrict expr to domain={\thisrow{pmax}}{4:4}, unbounded coords=discard] {data/transport_assembly_times_3d/transport_assembly_3d_quadrature_tets.csv}; \label{plots:transport_times_3d_quadrature_p_4}

            \addplot+[mark=square*, thick, dashed, tolcol1, mark options={tolcol1, solid}] table [x=no_eles, y=volume_assembly_timer, col sep=comma, restrict expr to domain={\thisrow{pmax}}{0:0}, unbounded coords=discard] {data/transport_assembly_times_3d/transport_assembly_3d_qfree_pruned_tets.csv}; \label{plots:transport_times_3d_qfree_p_0}
            
            \addplot+[mark=*, thick, dashed, tolcol2, mark options={tolcol2, solid}] table [x=no_eles, y=volume_assembly_timer, col sep=comma, restrict expr to domain={\thisrow{pmax}}{1:1}, unbounded coords=discard] {data/transport_assembly_times_3d/transport_assembly_3d_qfree_pruned_tets.csv}; \label{plots:transport_times_3d_qfree_p_1}
            
            \addplot+[mark=diamond*, thick, dashed, tolcol3, mark options={tolcol3, solid}] table [x=no_eles, y=volume_assembly_timer, col sep=comma, restrict expr to domain={\thisrow{pmax}}{2:2}, unbounded coords=discard] {data/transport_assembly_times_3d/transport_assembly_3d_qfree_pruned_tets.csv}; \label{plots:transport_times_3d_qfree_p_2}
            
            \addplot+[mark=triangle*, thick, dashed, tolcol4, mark options={tolcol4, solid}] table [x=no_eles, y=volume_assembly_timer, col sep=comma, restrict expr to domain={\thisrow{pmax}}{3:3}, unbounded coords=discard] {data/transport_assembly_times_3d/transport_assembly_3d_qfree_pruned_tets.csv}; \label{plots:transport_times_3d_qfree_p_3}
            
            \addplot+[mark=pentagon*, thick, dashed, tolcol5, mark options={tolcol5, solid}] table [x=no_eles, y=volume_assembly_timer, col sep=comma, restrict expr to domain={\thisrow{pmax}}{4:4}, unbounded coords=discard] {data/transport_assembly_times_3d/transport_assembly_3d_qfree_pruned_tets.csv}; \label{plots:transport_times_3d_qfree_p_4}
            
            \coordinate (top) at (rel axis cs:0,1);

        \nextgroupplot[xmode=log,
					   log basis x={10},
					   xticklabel={\pgfmathparse{10^(\tick)}\pgfmathprintnumber{\pgfmathresult}},
					   ymode=log,
					   log basis y={10},
					   ytick={1e-5,1e-4,1e-3,1e-2,1e-1,1e0,1e1,1e2},
					   ymin=1e-5, ymax=1e2,
					   xlabel=$|\mathcal{T}|$, 
					   axis background/.style={fill=gray!0}, 
					   legend pos=north west,
					   grid=both,
					   grid style={line width=.1pt, draw=gray!10},
   					   major grid style={line width=.2pt,draw=gray!50}]

            \addplot+[mark=square, thick, solid, tolcol1, mark options={tolcol1, solid}] table [x=no_eles, y=volume_assembly_timer, col sep=comma, restrict expr to domain={\thisrow{pmax}}{0:0}, unbounded coords=discard] {data/transport_assembly_times_3d/transport_assembly_3d_quadrature_agglomtets.csv}; 
            
            \addplot+[mark=o, thick, solid, tolcol2, mark options={tolcol2, solid}] table [x=no_eles, y=volume_assembly_timer, col sep=comma, restrict expr to domain={\thisrow{pmax}}{1:1}, unbounded coords=discard] {data/transport_assembly_times_3d/transport_assembly_3d_quadrature_agglomtets.csv}; 
            
            \addplot+[mark=diamond, thick, solid, tolcol3, mark options={tolcol3, solid}] table [x=no_eles, y=volume_assembly_timer, col sep=comma, restrict expr to domain={\thisrow{pmax}}{2:2}, unbounded coords=discard] {data/transport_assembly_times_3d/transport_assembly_3d_quadrature_agglomtets.csv}; 
            
            \addplot+[mark=triangle, thick, solid, tolcol4, mark options={tolcol4, solid}] table [x=no_eles, y=volume_assembly_timer, col sep=comma, restrict expr to domain={\thisrow{pmax}}{3:3}, unbounded coords=discard] {data/transport_assembly_times_3d/transport_assembly_3d_quadrature_agglomtets.csv}; 
            
            \addplot+[mark=pentagon, thick, solid, tolcol5, mark options={tolcol5, solid}] table [x=no_eles, y=volume_assembly_timer, col sep=comma, restrict expr to domain={\thisrow{pmax}}{4:4}, unbounded coords=discard] {data/transport_assembly_times_3d/transport_assembly_3d_quadrature_agglomtets.csv}; 

            \addplot+[mark=square*, thick, dashed, tolcol1, mark options={tolcol1, solid}] table [x=no_eles, y=volume_assembly_timer, col sep=comma, restrict expr to domain={\thisrow{pmax}}{0:0}, unbounded coords=discard] {data/transport_assembly_times_3d/transport_assembly_3d_qfree_pruned_agglomtets.csv}; 
            
            \addplot+[mark=*, thick, dashed, tolcol2, mark options={tolcol2, solid}] table [x=no_eles, y=volume_assembly_timer, col sep=comma, restrict expr to domain={\thisrow{pmax}}{1:1}, unbounded coords=discard] {data/transport_assembly_times_3d/transport_assembly_3d_qfree_pruned_agglomtets.csv}; 
            
            \addplot+[mark=diamond*, thick, dashed, tolcol3, mark options={tolcol3, solid}] table [x=no_eles, y=volume_assembly_timer, col sep=comma, restrict expr to domain={\thisrow{pmax}}{2:2}, unbounded coords=discard] {data/transport_assembly_times_3d/transport_assembly_3d_qfree_pruned_agglomtets.csv}; 
            
            \addplot+[mark=triangle*, thick, dashed, tolcol4, mark options={tolcol4, solid}] table [x=no_eles, y=volume_assembly_timer, col sep=comma, restrict expr to domain={\thisrow{pmax}}{3:3}, unbounded coords=discard] {data/transport_assembly_times_3d/transport_assembly_3d_qfree_pruned_agglomtets.csv}; 
            
            \addplot+[mark=pentagon*, thick, dashed, tolcol5, mark options={tolcol5, solid}] table [x=no_eles, y=volume_assembly_timer, col sep=comma, restrict expr to domain={\thisrow{pmax}}{4:4}, unbounded coords=discard] {data/transport_assembly_times_3d/transport_assembly_3d_qfree_pruned_agglomtets.csv}; 
            
            \coordinate (bot) at (rel axis cs:1,0);

    \end{groupplot}

	\path (top|-current bounding box.north)--
    	coordinate(legendpos)
    	(bot|-current bounding box.north);
	\matrix[
    	matrix of nodes,
    	anchor=south,
    	draw,
    	inner sep=0.2em,
    	draw
  	]at([yshift=1ex]legendpos)
  	{
        \ref{plots:transport_times_3d_quadrature_p_0} & Algorithm \ref{alg:stiffness_assembly_quadrature} ($p=0$) &[5pt] 
        \ref{plots:transport_times_3d_quadrature_p_1} & Algorithm \ref{alg:stiffness_assembly_quadrature} ($p=1$) &[5pt]
        \ref{plots:transport_times_3d_quadrature_p_2} & Algorithm \ref{alg:stiffness_assembly_quadrature} ($p=2$) \\ 
        \ref{plots:transport_times_3d_quadrature_p_3} & Algorithm \ref{alg:stiffness_assembly_quadrature} ($p=3$) &[5pt]
        \ref{plots:transport_times_3d_quadrature_p_4} & Algorithm \ref{alg:stiffness_assembly_quadrature} ($p=4$) &[5pt]
        \ref{plots:transport_times_3d_qfree_p_0} & Algorithm \ref{alg:stiffness_assembly_quadfree} ($p=0$) \\ 
        \ref{plots:transport_times_3d_qfree_p_1} & Algorithm \ref{alg:stiffness_assembly_quadfree} ($p=1$) &[5pt]
        \ref{plots:transport_times_3d_qfree_p_2} & Algorithm \ref{alg:stiffness_assembly_quadfree} ($p=2$) &[5pt]
        \ref{plots:transport_times_3d_qfree_p_3} & Algorithm \ref{alg:stiffness_assembly_quadfree} ($p=3$) \\ 
        & & [5pt] \ref{plots:transport_times_3d_qfree_p_4} & Algorithm \ref{alg:stiffness_assembly_quadfree} ($p=4$) &[5pt] \\
    };
        
	\end{tikzpicture}
	\caption{CPU times taken by Algorithm \ref{alg:stiffness_assembly_quadrature} (quadrature-based assembly) and Algorithm \ref{alg:stiffness_assembly_quadfree} (quadrature-free-based assembly) to evaluate the DGFEM element transport matrices on a sequence of meshes and for $0\le p\le4$. Left: total time taken on a sequence of tetrahedral meshes with $|\mathcal{T}|\in\{6,48,384,3072,24576\}$. Right: total time taken on a sequence of agglomerated meshes with $|\mathcal{T}|\in\{4,38,307,2457,19660\}$.}
    \label{fig:transport_times_3d_quad_vx_qfree}
\end{figure}

As a final example, we will compare the quadrature based and quadrature-free based assembly methods for the system matrix arising from DGFEM discretisation of the linear first-order transport problem posed in three spatial dimensions. We will consider sequences of tetrahedral and agglomerated tetrahedral meshes $\mathcal{T}$ and employ local polynomial bases $\mathbb{P}^p(\kappa)$ for each $\kappa\in\mathcal{T}$ with $0\le p\le 4$.

For standard tetrahedral meshes, the quadrature based method employs quadrature schemes consisting of $(p+2)^3$ points and weights on each tetrahedron; this ensures that the quadrature scheme exactly evaluates the element integrals appearing in \reff{eqn:dgfem_linear_transport}. The agglomerated tetrahedral meshes are formed by partitioning a fine mesh $\mathcal{T}_{fine}$ into polyhedral coarse-mesh elements $\kappa\in\mathcal{T}$ using METIS \cite{karypis1997metis}. The agglomeration strategy is chosen such that each coarse-mesh element $\kappa$ is formed from an average of 10 fine-mesh elements $\kappa_{fine}\in\mathcal{T}_{fine}$. The quadrature schemes on elements in $\mathcal{T}_{fine}$ are inherited by the coarse-mesh elements, ensuring that the integrals appearing in \reff{eqn:dgfem_linear_transport} can be evaluated exactly.

The quadrature-free based method is performed in two steps as before. On each element $\kappa\in\mathcal{T}$, the monomial integrals $\int_\kappa \x^{\bm{\alpha}} \d\x$ for $0\le|\bm{\alpha}|\le2p$ are computed once using an implementation of Algorithm~\ref{alg:monomial_integration} tailored to the three-dimensional setting. As before, we employ a first-vertex based pruning strategy to reduce the CPU time spent in Algorithm \ref{alg:monomial_integration}. The integrals in \reff{eqn:dgfem_linear_transport} are then evaluated using known decompositions of the integrands in terms of the monomial basis. The time taken to generate the coefficients in these decompositions is not included.

\input{includes/transport_times_3d/transport_assembly_3d_qfree_breakdown}

Figure \ref{fig:transport_times_3d_quad_vx_qfree} shows the CPU time taken by the quadrature based and quadrature-free based methods (Algorithms \ref{alg:stiffness_assembly_quadrature} and \ref{alg:stiffness_assembly_quadfree}, respectively) to assemble the global transport matrices arising from a DGFEM discretisation of a linear, constant-coefficient transport problem in three spatial dimensions. Both methods are tested on standard and agglomerated tetrahedral meshes for global polynomial degrees $0\le p\le 4$.

Both assembly methods are seen to scale linearly with the number of elements in the spatial mesh, as expected. For all tests recorded, the CPU time taken to assemble the system matrix using the quadrature-free method is consistently faster than the standard quadrature based approach by at most a constant multiplicative factor. For the tests performed on standard tetrahedral meshes, this multiplicative constant is between 2 and 3 for $p\ge1$. For agglomerated tetrahedral meshes, this multiplicative constant improves to at least 5 for $p\ge1$, with the quadrature based assembly taking almost 20 times longer than the quadrature-free based assembly in the case $p=4$.
This improvement in assembly time due to switching to a quadrature-free approach is expected to be greater on agglomerated meshes than tetrahedral meshes. Indeed, on a given element $\kappa\in\mathcal{T}$, the time complexity of the quadrature based method is $\bigO(n_{\kappa} p^{3d})=\bigO(n_{\kappa} p^{9})$, where $n_{\kappa}$ denotes the number of fine-mesh elements in $\mathcal{T}_{fine}$ that comprise $\kappa$. In contrast, the time complexity of the quadrature based method is $\bigO(\chi_1(\kappa)p^d + p^{3d})=\bigO(\chi_1(\kappa)p^3+p^9)$.

Figures \ref{fig:transport_assembly_3d_qfree_breakdown} and \ref{fig:transport_assembly_3d_qfree_breakdown_p_dep} present the breakdown of the contribution of line 4 and lines 7-14 of Algorithm~\ref{alg:stiffness_assembly_quadfree} to the total CPU time taken by the quadrature-free based algorithm. Both the integration and reconstruction phases are seen to scale linearly with the number of elements in the mesh; this is to be expected since lines 4 and lines 7-14 of Algorithm \ref{alg:stiffness_assembly_quadfree} are executed once for each element. The reconstruction \reff{eqn:element_matrix_reconstruction_formula} performed on lines 7-14 of Algorithm \ref{alg:stiffness_assembly_quadfree} is seen to scale much faster as a function of $p$ than the integration of the monomial basis via Algorithm \ref{alg:monomial_integration}; this is evidenced by the greater separation of the data corresponding to the CPU times for each polynomial degree $p$. However, the time taken in the reconstruction phase is seen to be independent of the geometry of the mesh elements, whereas a greater amount of CPU time is required to perform the integration phase on agglomerated tetrahedral meshes.

\iffalse
\subsubsection{Effect of pruning in Algorithm \ref{alg:monomial_integration}}

\begin{table}[t!]
    \centering
    \begin{tabular}{cc|ccccc}
        \multicolumn{2}{c|}{$|\mathcal{T}|$} & 6 & 48 & 384 & 3072 & 24576 \\ \hline
        \multirow{3}{*}{$p=0$} & Unpruned & 2.4926E-05 & 1.0816E-04 & 8.1476E-04 & 6.8287E-03 & 4.8521E-02 \\
        & Pruned & 1.9337E-05 & 7.1661E-05 & 5.1822E-04 & 4.0497E-03 & 3.1205E-02 \\ 
        & Ratio & 0.7758 & 0.6625 & 0.6360 & 0.5930 & 0.6431 \\ \hline
        \multirow{3}{*}{$p=1$} & Unpruned & 3.5647E-05 & 2.0676E-04 & 1.5616E-03 & 1.3476E-02 & 9.5595E-02 \\
        & Pruned & 2.1014E-05 & 9.8906E-05 & 7.1403E-04 & 5.6071E-03 & 4.3727E-02 \\
        & Ratio & 0.5895 & 0.4784 & 0.4572 & 0.4161 & 0.4574 \\ \hline
        \multirow{3}{*}{$p=2$} & Unpruned & 6.3102E-05 & 4.2202E-04 & 3.2991E-03 & 2.5008E-02 & 1.9898E-01 \\
        & Pruned & 3.1368E-05 & 1.5550E-04 & 1.2420E-03 & 9.1364E-03 & 7.2012E-02 \\
        & Ratio & 0.4971 & 0.3685 & 0.3765 & 0.3653 & 0.3619 \\ \hline
        \multirow{3}{*}{$p=3$} & Unpruned & 1.1175E-04 & 8.2232E-04 & 5.9819E-03 & 4.7363E-02 & 3.7750E-01 \\
        & Pruned & 4.4362E-05 & 2.5484E-04 & 1.8026E-03 & 1.3952E-02 & 1.1175E-01 \\
        & Ratio & 0.3970 & 0.3099 & 0.3013 & 0.2946 & 0.2960 \\ \hline
        \multirow{3}{*}{$p=4$} & Unpruned & 1.8889E-04 & 1.4290E-03 & 1.0750E-02 & 8.6211E-02 & 6.8734E-01 \\
        & Pruned & 6.2170E-05 & 2.8795E-04 & 3.0248E-03 & 2.3477E-02 & 1.8765E-01 \\
        & Ratio & 0.3291 & 0.2015 & 0.2814 & 0.2723 & 0.2730 \\
    \end{tabular}
    \caption{CPU times of unpruned and pruned versions of Algorithm \ref{alg:monomial_integration} (executed on line 4 of Algorithm \ref{alg:stiffness_assembly_quadfree}) for the assembly of the element matrices $\mathbf{A}_\kappa$ for tetrahedral elements.}
    \label{tab:transport_assembly_monom_int_pruning_tets}
\end{table}

\begin{table}[t!]
    \centering
    \begin{tabular}{cc|ccccc}
        \multicolumn{2}{c|}{$|\mathcal{T}|$} & 4 & 38 & 307 & 2457 & 19660 \\ \hline
        \multirow{3}{*}{$p=0$} & Unpruned & 4.5272E-05 & 3.5574E-04 & 2.6056E-03 & 2.1919E-02 & 1.7550E-01 \\
        & Pruned & 3.7088E-05 & 2.8895E-04 & 2.0201E-03 & 1.6717E-02 & 1.3413E-01 \\ 
        & Ratio & 0.8192 & 0.8123 & 0.7753 & 0.7627 & 0.7643 \\ \hline
        \multirow{3}{*}{$p=1$} & Unpruned & 8.7692E-05 & 7.7001E-04 & 5.5442E-03 & 4.3656E-02 & 3.5131E-01 \\
        & Pruned & 5.4265E-05 & 4.5130E-04 & 3.1994E-03 & 2.6491E-02 & 2.0647E-01 \\ 
        & Ratio & 0.6188 & 0.5861 & 0.5771 & 0.6068 & 0.5877 \\ \hline
        \multirow{3}{*}{$p=2$} & Unpruned & 1.9519E-04 & 1.5893E-03 & 1.1656E-02 & 9.1985E-02 & 7.4515E-01 \\
        & Pruned & 9.1346E-05 & 7.6676E-04 & 5.5029E-03 & 4.4103E-02 & 3.5562E-01 \\ 
        & Ratio & 0.4680 & 0.4825 & 0.4721 & 0.4795 & 0.4772 \\ \hline
        \multirow{3}{*}{$p=3$} & Unpruned & 3.4222E-04 & 3.0026E-03 & 2.3204E-02 & 1.8303E-01 & 1.4867E+00 \\
        & Pruned & 1.6222E-04 & 1.2524E-03 & 9.9395E-03 & 7.9390E-02 & 6.4008E-01 \\ 
        & Ratio & 0.4740 & 0.4171 & 0.4284 & 0.4338 & 0.4305 \\ \hline
        \multirow{3}{*}{$p=4$} & Unpruned & 6.0944E-04 & 5.2648E-03 & 4.1087E-02 & 3.2613E-01 & 2.6380E+00 \\
        & Pruned & 2.4756E-04 & 1.9114E-03 & 1.5453E-02 & 1.2295E-01 & 1.0089E+00 \\ 
        & Ratio & 0.4062 & 0.3631 & 0.3761 & 0.3770 & 0.3824 \\
    \end{tabular}
    \caption{CPU times of unpruned and pruned versions of Algorithm \ref{alg:monomial_integration} (executed on line 4 of Algorithm \ref{alg:stiffness_assembly_quadfree}) for the assembly of the element matrices $\mathbf{A}_\kappa$ for agglomerated tetrahedral elements.}
    \label{tab:transport_assembly_monom_int_pruning_agglomtets}
\end{table}

We also study the effect of implementing Algorithm \ref{alg:monomial_integration} with and without pruning as described in Section \ref{section:monomial_integration_quad_free_analysis_subsec}. Tables \ref{tab:transport_assembly_monom_int_pruning_tets} and \ref{tab:transport_assembly_monom_int_pruning_agglomtets} show the total CPU times taken by the unpruned and pruned versions of Algorithm \ref{alg:monomial_integration} applied to tetrahedral and agglomerated tetrahedral meshes respectively. An additional quantity, computed as the ratio of the CPU time taken by the pruned algorithm against the unpruned algorithm, is also reported; values of this ratio less than 1 indicate that Algorithm \ref{alg:monomial_integration} with pruning computes the integral set $\ifun(\polytope,\mathcal{J})$ faster than the same algorithm without pruning.

For each $0\le p\le 4$, it is observed that the ratio of CPU time taken by the pruned algorithm against the unpruned algorithm remains roughly constant. For a fixed number of elements, this ratio decreases as $p$ increases. It is expected that this ratio continues to decrease as $p\rightarrow\infty$ but remains bounded from below by a constant - for a single element $\kappa$, this constant is expected to depend on the number of nodes and edges of the graph $G(\kappa)$ defined in Section \ref{section:monomial_integration_quad_free_analysis_subsec} before and after pruning. While pruning accelerates the assembly of the element matrix contributions $\mathbf{A}_\kappa$ for both tetrahedral and agglomerated tetrahedral elements, a greater improvement in assembly time is observed for tetrahedral meshes.

\fi

\section{Conclusions} \label{sec:conclusions}

In this article we have analysed the computational complexity of computing the integral of families of polynomial spaces over general polytopic domains.
Starting from the ideas developed in \cite{chin2015numerical,lasserre1998integration} for the integration of homogeneous functions, we have demonstrated that the time and space complexities required to integrate families of monomial functions are dependent on three factors: the ambient dimension of the polytopic domain; the size of the requested set of monomial integrals; and the size of a directed graph related to the polytopic domain. In the case of polygonal or polyhedral geometries, the monomial integration algorithm is shown to scale linearly with the number of graph edges. This algorithm was applied to the computation of element integrals arising in the DGFEM discretisation of the linear transport problem. We have shown that, by decomposing the integrand into a linear combination of monomial functions, these integrals could be evaluated at speeds comparable to methods based on employing quadrature schemes with a minimal number of points and weights. In comparison to quadrature based methods employing a sub-tessellation, the quadrature-free approach is seen to both significantly accelerate the assembly of the system matrix and also scale independently of the element geometry for sufficiently-high polynomial degrees.

\section*{Acknowledgments}
PH and MEH acknowledge the financial support of the EPSRC (grant EP/R030707/1). PH also acknowledges the financial support of the MRC (grant MR/T017988/1).

\bibliographystyle{acm}
\bibliography{references}

\begin{thebibliography}{10}

\bibitem{Antonietti2016}
{\sc Antonietti, P.~F., Cangiani, A., Collis, J., Dong, Z., Georgoulis, E.~H.,
  Giani, S., and Houston, P.}
\newblock {\em Review of Discontinuous Galerkin Finite Element Methods for
  Partial Differential Equations on Complicated Domains}.
\newblock Springer International Publishing, Cham, 2016, pp.~281--310.

\bibitem{antonietti2018fast}
{\sc Antonietti, P.~F., Houston, P., and Pennesi, G.}
\newblock {Fast numerical integration on polytopic meshes with applications to
  discontinuous Galerkin finite element methods}.
\newblock {\em J. Sci. Comput. 77}, 3 (2018), 1339--1370.

\bibitem{doi:10.1142/S0218202512500492}
{\sc Beir\~{a}o Da~Veiga, L., Brezzi, F., Cangiani, A., Manzini, G., Marini,
  L.~D., and Russo, A.}
\newblock Basic principles of virtual element methods.
\newblock {\em Math. Models Methods Appl. Sci. 23}, 01 (2013), 199--214.

\bibitem{bueler2000exact}
{\sc B{\"u}eler, B., Enge, A., and Fukuda, K.}
\newblock Exact volume computation for polytopes: a practical study.
\newblock In {\em Polytopes—combinatorics and computation\/} (2000),
  Springer, pp.~131--154.

\bibitem{CaDoGeHo2017}
{\sc Cangiani, A., Dong, Z., Georgoulis, E.~H., and Houston, P.}
\newblock {\em {$hp$}-{V}ersion discontinuous Galerkin methods on polygonal and
  polyhedral meshes}.
\newblock SpringerBriefs in Mathematics. Springer International Publishing,
  2017.

\bibitem{cangiani2014hp}
{\sc Cangiani, A., Georgoulis, E.~H., and Houston, P.}
\newblock hp-{V}ersion discontinuous {G}alerkin methods on polygonal and
  polyhedral meshes.
\newblock {\em Math. Models Methods Appl. Sci. 24}, 10 (2014), 2009--2041.

\bibitem{chin2015numerical}
{\sc Chin, E.~B., Lasserre, J.~B., and Sukumar, N.}
\newblock {Numerical integration of homogeneous functions on convex and
  nonconvex polygons and polyhedra}.
\newblock {\em Comp. Mech. 56\/} (2015), 967--981.

\bibitem{chin2020efficient}
{\sc Chin, E.~B., and Sukumar, N.}
\newblock {An efficient method to integrate polynomials over polytopes and
  curved solids}.
\newblock {\em Comput. Aided Geom. Design 82\/} (2020), 101914.

\bibitem{hho_book_2021}
{\sc Cicuttin, M., Ern, A., and Pignet, N.}
\newblock {\em Hybrid High-Order Methods. A Primer with Applications to Solid
  Mechanics}.
\newblock SpringerBriefs in Mathematics. Springer International Publishing,
  2021.

\bibitem{doi:10.1137/20M1350984}
{\sc Dong, Z., Georgoulis, E.~H., and Kappas, T.}
\newblock {GPU}-accelerated discontinuous {G}alerkin methods on polytopic
  meshes.
\newblock {\em SIAM J. Sci. Comput. 43}, 4 (2021), C312--C334.

\bibitem{duffy1982quadrature}
{\sc Duffy, M.~G.}
\newblock {Quadrature over a pyramid or cube of integrands with a singularity
  at a vertex}.
\newblock {\em SIAM J. Numer. Anal. 19}, 6 (1982), 1260--1262.

\bibitem{grunbaum1967convex}
{\sc Gr{\"u}nbaum, B., Klee, V., Perles, M.~A., and Shephard, G.~C.}
\newblock {\em Convex polytopes}, vol.~16.
\newblock Springer, 1967.

\bibitem{houston2023efficient}
{\sc Houston, P., Hubbard, M.~E., Radley, T.~J., Sutton, O.~J., and Widdowson,
  R. S.~J.}
\newblock Efficient high-order space-angle-energy polytopic discontinuous
  {G}alerkin finite element methods for linear {B}oltzmann transport, Submitted
  for publication.

\bibitem{kaibel2002computing}
{\sc Kaibel, V., and Pfetsch, M.~E.}
\newblock Computing the face lattice of a polytope from its vertex-facet
  incidences.
\newblock {\em Comp. Geom. 23}, 3 (2002), 281--290.

\bibitem{karypis1997metis}
{\sc Karypis, G., and Kumar, V.}
\newblock {METIS: A software package for partitioning unstructured graphs,
  partitioning meshes, and computing fill-reducing orderings of sparse
  matrices}.

\bibitem{Knuth81}
{\sc Knuth, D.~E.}
\newblock {\em The Art of Computer Programming}, vol.~2: Seminumerical
  Algorithms.
\newblock Addison-Wesley, 1981.

\bibitem{lasserre1998integration}
{\sc Lasserre, J.}
\newblock Integration on a convex polytope.
\newblock {\em Proceedings of the American Mathematical Society 126}, 8 (1998),
  2433--2441.

\bibitem{lasserre1999integration}
{\sc Lasserre, J.}
\newblock {Integration and homogeneous functions}.
\newblock {\em Proceedings of the American Mathematical Society 127}, 3 (1999),
  813--818.

\bibitem{LyMo77}
{\sc Lyness, J.~N., and Monegato, G.}
\newblock Quadrature rules for regions having regular hexagonal symmetry.
\newblock {\em SIAM J. Numer. Anal. 14}, 2 (1977), 283 -- 295.

\bibitem{mousavi2010generalized}
{\sc Mousavi, S., Xiao, H., and Sukumar, N.}
\newblock {Generalized Gaussian quadrature rules on arbitrary polygons}.
\newblock {\em Internat. J. Numer. Methods Engrg. 82}, 1 (2010), 99--113.

\bibitem{MoSu11}
{\sc Mousavi, S.~E., and Sukumar, N.}
\newblock Numerical integration of polynomials and discontinuous functions on
  irregular convex polygons and polyhedrons.
\newblock {\em Comput. Mech. 47}, 5 (2011), 535 -- 554.

\bibitem{simon1994mathematics}
{\sc Simon, C.~P., and Blume, L.}
\newblock {\em {Mathematics for economists}}, vol.~7.
\newblock Norton New York, 1994.

\bibitem{stroud1971approximate}
{\sc Stroud, A.}
\newblock {Approximate calculation of multiple integrals. Prentice-Hall series
  in automatic computation}.

\bibitem{SuWa13}
{\sc Sudhakar, Y., and Wall, W.~A.}
\newblock Quadrature schemes for arbitrary convex/concave volumes and
  integration of weak form in enriched partition of unity methods.
\newblock {\em Comput. Methods Appl. Mech. Engrg. 258}, 1 (2013), 39 -- 54.

\bibitem{sukumar2004conforming}
{\sc Sukumar, N., and Tabarraei, A.}
\newblock {Conforming polygonal finite elements}.
\newblock {\em Internat. J. Numer. Methods Engrg. 61}, 12 (2004), 2045--2066.

\bibitem{taylor1996partial}
{\sc Taylor, M.~E.}
\newblock {\em {Partial differential equations. 1, Basic theory}}.
\newblock Springer, 1996.

\bibitem{XiGi10}
{\sc Xiao, H., and Gimbutas, Z.}
\newblock A numerical algorithm for the construction of efficient quadrature
  rules in two and higher dimensions.
\newblock {\em Comput. Math. Appl. 59}, 2 (January 2010), 663 -- 676.

\end{thebibliography}

\end{document}